\pdfoutput=1
\documentclass[11pt,reqno]{amsart}
\usepackage{rotating}
\usepackage{bbm}
\usepackage{mathrsfs}
\usepackage{stmaryrd}
\usepackage{amscd,latexsym,amsthm,amsfonts,amssymb,amsmath,amsxtra}

\usepackage[hyperfootnotes=false, colorlinks, citecolor=RoyalBlue, linkcolor=Bittersweet, urlcolor=link]{hyperref}
\usepackage{color}
\usepackage[usenames,dvipsnames]{xcolor}
\usepackage{mathpazo}
\usepackage[mathcal]{eucal}
\usepackage{graphicx}
\usepackage[all]{xy}

\usepackage{setspace}
\usepackage[new]{old-arrows}

\usepackage{tikz}
\usetikzlibrary{calc, intersections}
\usepackage{tikz-cd}

\newcommand{\FL}{{\mathbf{L}}}
\newcommand{\FH}{{\mathbf{H}}} \newcommand{\FG}{{\mathbf{G}}}

\newcommand{\FP}{{\mathbf{P}}}\newcommand{\FQ}{{\mathbf{Q}}}
\newcommand{\FM}{{\mathbf{M}}}
\newcommand{\FN}{{\mathbf{N}}}

\newcommand{\FX}{{\mathbf{X}}}
\newcommand{\FY}{{\mathbf{Y}}}
\newcommand{\FW}{{\mathbf{W}}}\newcommand{\FV}{{\mathbf{V}}}

\newcommand{\Fq}{{\mathbf {q}}}

\newcommand{\BA}{{\mathbb {A}}} 
\newcommand{\BC}{{\mathbb {C}}}

 \newcommand{\BN}{{\mathbb {N}}}

\newcommand{\BW}{{\mathbb {W}}} \newcommand{\BX}{{\mathbb {X}}}
\newcommand{\BY}{{\mathbb {Y}}} \newcommand{\BZ}{{\mathbb {Z}}}

\newcommand{\CC}{{\mathcal {C}}} 
 \newcommand{\CF}{{\mathcal {F}}}
 
\newcommand{\CI}{{\mathcal {I}}} 
 
\newcommand{\CM}{{\mathcal {M}}} 
\newcommand{\CO}{{\mathcal {O}}} \newcommand{\CP}{{\mathcal {P}}}
 
\newcommand{\CS}{{\mathcal {S}}} 
 \newcommand{\CV}{{\mathcal {V}}}
\newcommand{\CW}{{\mathcal {W}}} \newcommand{\CX}{{\mathcal {X}}}

 \newcommand{\RH}{{\mathrm {H}}}

\newcommand{\an}{{\mathrm{an}}}

\newcommand{\bc}{{\mathrm{bc}}}

\newcommand{\rc}{{\mathrm{c}}}

\newcommand{\disc}{{\mathrm{disc}}}
\newcommand{\diag}{{\mathrm{diag}}}\renewcommand{\d}{{\mathrm{d}}}

\newcommand{\End}{{\mathrm{End}}} 
 \newcommand{\Ext}{\mathrm{Ext}}

\newcommand{\Gal}{{\mathrm{Gal}}} 
\newcommand{\GSp}{{\mathrm{GSp}}} \newcommand{\PGSp}{{\mathrm{PGSp}}}
\newcommand{\GL}{{\mathrm{GL}}}
\newcommand{\Hom}{{\mathrm{Hom}}}

\newcommand{\id}{{\mathrm{id}}}
\newcommand{\Ind}{{\mathrm{Ind}}} 
\newcommand{\Irr}{{\mathrm{Irr}}}

\newcommand{\M}{{\mathrm{M}}}

\newcommand{\Mp}{{\mathrm{Mp}}}  
 
 \newcommand{\N}{{\mathrm{N}}}

\renewcommand{\O}{{\mathrm{O}}}
 \newcommand{\ONB}{{\mathrm{ONB}}}

\renewcommand{\Re}{{\mathrm{Re}}}
\newcommand{\Rep}{{\mathrm{Rep}}} \newcommand{\rat}{{\mathrm{ra}}}

 \newcommand{\sm}{{\mathrm{sm}}}

\newcommand{\SL}{{\mathrm{SL}}}
 
\newcommand{\SO}{{\mathrm{SO}}}\newcommand{\Sp}{{\mathrm{Sp}}}

\newcommand{\Stab}{{\mathrm{Stab}}}

    \newcommand{\temp}{{\mathrm{temp}}}

\newcommand{\tr}{{\mathrm{tr}}}

\newcommand{\R}{{\mathrm{R}}}

\newcommand{\vol}{{\mathrm{vol}}}

\newcommand{\wt}{\widetilde}

\newcommand{\pair}[1]{\langle {#1} \rangle}

\newcommand{\norm}[1]{\|{#1}\|}

\newcommand{\incl}{\hookrightarrow}

\newcommand{\sk}{\medskip}
\newcommand{\lra}{\longrightarrow}
\newcommand{\ra}{\rightarrow} 

\newcommand{\bs}{\backslash}
\newcommand{\lhra}{\longhookrightarrow}

\newcommand{\s}{\sk\noindent}

\theoremstyle{plain}
\newtheorem{thm}{Theorem}[section] \newtheorem{cor}[thm]{Corollary}
\newtheorem{lem}[thm]{Lemma}  \newtheorem{prop}[thm]{Proposition}

\theoremstyle{definition}
 \newtheorem{defn}[thm]{Definition}
\newtheorem{rem}[thm]{Remark} 
\newtheorem{eg}[thm]{Example}

\numberwithin{equation}{section}

\title{Local theta correspondence and Galois periods}
\author{Chong Zhang}
\begin{document}
\date{}
\maketitle

\begin{abstract}
We study the behavior of Galois periods under the local theta correspondence for even orthogonal and symplectic groups. Specifically, we compare their multiplicities and construct explicit transfer maps. Furthermore, we establish both an adjoint relation and a relative character relation for these periods.
\end{abstract}

\tableofcontents

\section{Introduction}

\subsection{Overview}
The local aspect of the relative Langlands program, systematically initiated by Sakellaridis and Venkatesh \cite{sv} and further developed by Ben-Zvi, Sakellaridis and Venkatesh \cite{bzsv}, concerns the properties of local periods. One powerful tool for studying relations of certain periods between members of a reductive dual pair is the local theta correspondence, a method exposited comprehensively in \cite{gan19, gan25}. Among various periods, the Galois period is of particular importance. For this period, Prasad \cite{pra} formulated a  concrete conjecture relating its multiplicity to the $L$-parameter. Prasad's conjecture has been explored in a series of works by Lu for classical groups of small rank (e.g., \cite{lu20a, lu20b}). Building on these developments, this paper investigates the behavior of Galois periods under the local theta correspondence between even orthogonal and symplectic groups.

\subsection{Main results}
Let $F$ be a non-archimedean local field of characteristic 0, and $E$ a quadratic field extension of $F$. Denote by $\Gal(E/F)$ the Galois group of this extension, and by $\iota$ its non-trivial element . Fix an element $\tau\in E^\times$ such that $\tr_{E/F}(\tau)=0$, where $\tr_{E/F}: E\ra F$ denotes the trace map. Let $\psi_F$ be a non-trivial additive character of $F$. We define the additive character $\psi$ of $E$ by 
$\psi:=\psi_F\circ\tr_{E/F}$, and its $\tau$-twist by $\psi_\tau(x):=\psi(\tau x)$ for $x\in E$. 

Let $G$ be a reductive group over $F$. The automorphism $\iota$ then acts naturally as an involution on $G(E)$, with fixed-point subgroup $G(F)$. For a smooth representation $\pi$ of $G(E)$, the space $\Hom_{G(F)}(\pi,\BC)$ of the $G(F)$-invariant linear forms is called the space of (local) {\em Galois periods}. We call $\dim\Hom_{G(F)}(\pi,\BC)$ the {\em multiplicity} of $\pi$. By a general theorem of Delorme \cite{del}, if $\pi$ is irreducible, then it is of finite multiplicity, though not necessarily of multiplicity one.  For Galois periods
$\alpha$ and $\alpha'$ of $\Hom_{G(F)}(\pi,\BC)$, when $\pi$ is irreducible, one can define the relative character $\Phi_{\pi,\alpha,\alpha'}$ (see Sect. \ref{subsec rel-char} for the precise definition). 

Let $V_F$ be a quadratic space of dimension $2m$, and $W_F$ a symplectic space of dimension $2n$ over $F$. Denote by $V$ and $W$ the base change of $V_F$ and $W_F$ to $E$, and by $\O(V_F)$, 
$\O(V)$, $\Sp(W_F)$ and $\Sp(W)$ the corresponding isometry groups. Then $\O(V)$ and $\Sp(W)$ form a reductive dual pair.  Given an irreducible admissible representation $\pi$ of $\Sp(W)$, let $\Theta_{V,W,\psi_\tau}(\pi)$ denote the (big) theta lift of $\pi$ to $\O(V)$; it is a smooth representation of $\O(V)$ of finite length. 

To state our results in a simpler form, we focus on the direction from $\Sp(W)$ to $\O(V)$ in the local theta correspondence and, for simplicity, assume that $\pi$ is supercuspidal and $\Theta_{V,W,\psi_\tau}(\pi)$ is the first occurrence. Under this assumption, a fundamental result of Kudla \cite{kud86} asserts that $$\sigma:=\Theta_{V,W,\psi_\tau}(\pi)$$ is irreducible and supercuspidal. We also note that in this setting, $\pi$ is unitary and its contragredient is isomorphic to its complex conjugate.

Our first result compares the multiplicities of $\pi$ and $\sigma$.

\begin{thm}[Theorem \ref{thm symp-orth}]
If $m>n$, then $\dim \Hom_{\Sp(W_F)}(\pi,\BC)=\dim\Hom_{\O(V_F)}(\sigma,\BC)$.
\end{thm}

The second result establishes an explicit transfer map between the spaces of Galois periods. We define the \emph{base change doubling zeta integral} which induces a linear map $$T_{W_F}: \Hom_{\Sp(W_F)}(\pi,\BC)\lra\Hom_{\O(V_F)}(\sigma,\BC).$$

\begin{thm}[Theorem \ref{thm transf1}] if $m>n$, the map $T_{W_F}$ is an isomorphism.
 \end{thm}

Dually, we can define a linear map in the opposite direction:
$$T_{V_F}: \Hom_{\O(V_F)}(\sigma,\BC)\lra\Hom_{\Sp(W_F)}(\pi,\BC).$$ Let $\pair{\textrm{-},\textrm{-}}_{X_{W_F}}$ and
 $\pair{\textrm{-},\textrm{-}}_{X_{V_F}}$ be the inner products on 
$\Hom_{\Sp(W_F)}(\pi,\BC)$ and $\Hom_{\O(V_F)}(\sigma,\BC)$,   respectively (defined in Sect. \ref{subsec adj rel}).
Our third result shows that $T_{W_F}$ and $T_{V_F}$ are adjoint to each other.

\begin{thm}[Theorem \ref{thm adjoint}] 
For any $\alpha\in\Hom_{\Sp(W_F)}(\pi,\BC)$ and 
$\beta\in\Hom_{\O(V_F)}(\sigma,\BC)$, we have
$$\pair{T_{W_F}(\alpha),\beta}_{X_{V_F}}=\pair{T_{V_F}(\beta),\alpha}_{X_{W_F}}.$$ 
\end{thm}

Our final result  is a relative character identity that connects the periods and their transfers.

\begin{thm}[Theorem \ref{thm rel-char}]
For all $\alpha\in\Hom_{\Sp(W_F)}\left(\pi,\BC\right)$ and $\beta\in\Hom_{\O(V_F)}\left(\sigma,\BC\right)$, and for test functions $f$ on $\Sp(W_F)\bs\Sp(W)$ and $f'$ on $\O(V_F)\bs\O(V)$ that  correspond to each other in the sense of Definition \ref{def function}, we have
$$\Phi_{\pi,\alpha,T_{V_F}(\beta)}(f)=\Phi_{\sigma,\beta,T_{W_F}(\alpha)}(f').$$ 
\end{thm}

\subsection{The method} 
 The primary tool for proving the above results is a variant of the classical doubling method of Piatetski-Shapiro and Rallis \cite{gpsr}, which we term the {\em base change doubling method}. While Lu has successfully applied this method in his study of the multiplicities of Galois periods, it possesses certain limitations. The present work develops the method further and overcomes these limitations, as we will now explain. 

\subsubsection{}  Consider the case of orthogonal groups as an example. In the base change doubling method,  
we can view $V$ as a vector space over $F$ via restriction of scalars; denote this $F$-space by $\FV$. Let $q_F$ be the quadratic form on $V_F$ and $q$ the quadratic form on $V$. Then $\Fq:=\tr_{E/F}\circ q$ defines an $F$-quadratic form on $\FV$. Since $\FV=V_F\oplus\tau V_F$ as $F$-vector spaces, $\FV$ is a ``doubled space" of $V_F$ . This is precisely the base change doubled space used by Lu. 

In the classical doubling method, the doubled space of $V_F$ is $V_F^\square=V_F\oplus V_F^-$, where $V_F^-$ denotes the space $V_F$ equipped with the form $-q_F$. A key feature is that $V^\square$ is split, possessing the Lagrangian subspace $V_F^\Delta=\{(x,x)\mid x\in V_F\}$, which plays an essential role in applications. This yields  natural embeddings $$\O(V_F)\lhra\O(V_F)\times\O(V_F)\lhra \O(V_F^\square).$$  However, $\FV$ is not split in general. Consequently, Lu had to assume that $V$ itself is split to ensure $\FV$ split. Our key idea is to first twist the quadratic form on $V$ by $\tau$ before applying the doubling construction. Precisely, let $V_\tau$ be the quadratic space with underlying space $V$ and form $q_\tau:=\tau q$. Let $\FV_\tau$ be the corresponding $F$-space via restriction of scalars, equipped with the form $\Fq_\tau:=\tr_{E/F}\circ q_\tau$. One readily checks that $\FV_\tau$ is split and contains $V_F$ as a Lagrangian. Identifying $\O(V)$ naturally with $\O(V_\tau)$, we obtain embeddings $$\O(V_F)\lhra\O(V)\lhra \O(\FV_\tau).$$ 

 We remark that the $\tau$-twist serves a dual purpose: it ensures the doubled space is split, and is moreover compatible with the data used in the local theta correspondence—specifically, the lift $\Theta_{V,W,\psi_\tau}$ rather than $\Theta_{V,W,\psi}$.

\subsubsection{}  Equipped with the base change doubling method, we apply the base change seesaw identity and analyze the structure of the degenerate principal series to compare the multiplicities of $\pi$ and $\sigma$. Beyond this multiplicity comparison, which was also studied by Lu, our work establishes new results on the explicit transfer of Galois periods and their relative character relation. 
The approach to these results is inspired, in part, by ideas presented in talks of Gan on how to transfer Harish-Chandra characters via the Weil representation, and also by the work of Gan and Wan \cite{gw21} on relative character identities
for the spherical variety $\SO_{n-1}\bs\SO_n$ and the Whittaker variety $(N,\psi)\bs\SL_2$.
A key technical tool for achieving these results is the base change doubling zeta integral, a variant of the classical doubling zeta integral. We develop and establish several of its elementary properties in this paper.

\subsection{Few remarks}
We conclude by noting some limitations and directions not covered in this paper.
 
\subsubsection{} Our results are established for representations with special properties, such as tempered, square-integrable, or supercuspidal representations. Extending these results to more general representations remains a natural and desirable direction for future work.

\subsubsection{} This paper does not discuss the local theta correspondence between odd orthogonal and metaplectic groups. For that setting, a relation between multiplicities has been obtained by Lu \cite{lu20c}.  We expect that the results of the present article have direct analogues there, and the proofs should carry over with only minor adjustments.

\subsubsection{} 
The classical doubling zeta integral belongs to the class of Rankin–Selberg integrals and serves to represent the standard $L$-functions of classical groups. Its fundamental analytic properties, including absolute convergence, meromorphic continuation, functional equation, and the unramified calculation, were established in the foundational works of Piatetski-Shapiro and Rallis \cite{gpsr}, Lapid and Rallis \cite{lr}, and Yamana \cite{yam}.

In contrast, the base change doubling zeta integral studied in this paper introduces a key new feature: the relevant linear functional is not necessarily of multiplicity one, unlike in the classical setting. This fundamental difference introduces significant challenges in developing a parallel, complete analytic theory for the base change doubling zeta integral. Consequently, the aforementioned analytic properties (convergence, meromorphic continuation, etc.) are not established for our integral here.

Nevertheless, we believe that the base change doubling zeta integral is of intrinsic interest and likely bears a close relation to certain $L$-functions, analogous to its classical counterpart.

\subsubsection{}
The classical global doubling zeta integral, combined with the Siegel–Weil formula, provides a powerful tool for studying the non-vanishing of global theta lifts. For some special dual pairs of small rank, Lu \cite{lu21} established relations between automorphic Galois periods under global theta lifting.

The theory developed in this paper for the base change doubling method readily extends to the setting of number fields. This extension allows one to define a global base change doubling zeta integral, whose unfolding naturally involves global Galois periods. With the aid of the Siegel–Weil formula, this framework can be used to establish relations for the non-vanishing of global Galois periods, not only for general even orthogonal-symplectic dual pairs but also for odd orthogonal-metaplectic dual pairs.

\subsection{Structure of this paper} Let $G$ denote either $\Sp(W)$ or $\O(V)$, and $\FG$ denote $\Sp(\FW)$ or $\O(\FV_\tau)$, respectively. The paper is organized as follows.

    Section \ref{sec doubling} provides a systematic treatment of the base change doubling construction. We analyze in detail the orbits and stabilizers for the action of $G$ on the flag variety $\FP\bs\FG$, supplying proofs and details for some results that appear scattered in the works of Lu.

    Section \ref{sec weil} reviews Weil representations and base change seesaw dual pairs. We pay careful attention to the choice of splittings to ensure that they are compatible, and discuss the concrete models of Weil representations and the transitions between them, which are crucial for our later arguments.

    Section \ref{sec theta} recalls foundational results in the theory of local theta correspondence and states the base change seesaw identity.

    Section \ref{sec deg} studies the structure of the relevant degenerate principal series. We examine its structure both as a $\FG$-module—where it relates to the big theta lift of the trivial representation—and as a $G$-module, where it relates to certain induced representations.

    Section \ref{sec mult} compares the multiplicities of representations related by the local theta correspondence, using the base change seesaw identity and the structure of the degenerate principal series established in the previous sections.

    Section \ref{sec transfer} introduces the base change doubling zeta integral and proves its absolute convergence for square-integrable and tempered representations. We then define test functions, construct explicit transfer operators for Galois periods, and finally establish an adjoint relation between Galois periods related by the transfer operators as well as the corresponding relative character relation.

\subsection{Notation and convention} Throughout out the paper, $F$ denotes a non-archimedean local field of characteristic 0 and residue characteristic $p$, and $E$ a quadratic field extension of $F$.  We write $\R_{E/F}(\textrm{-})$ for the restriction of scalars with respect to $F\subset E$. Let $\CO_F$ be the valuation ring of $F$.

Let $|\textrm{-}|_F$ and $|\textrm{-}|_E$ be the normalized absolute values on $F$ and $E$, respectively, so that $|x|_E=|\N_{E/F}(x)|_F$ for $x\in E$, where $\N_{E/F}:E\ra F$ is the norm map. Denote by $\iota$ the non-trivial automorphism of $\Gal(E/F)$. We fix a non-zero element $\tau\in E$ such that $\tr_{E/F}(\tau)=0$. 

Fix a non-trivial additive character $\psi_F$ of $F$. Define the additive character $\psi$ of $E$ by $\psi:=\psi_F\circ\tr_{E/F}$, and its $\tau$-twist by $\psi_\tau(x):=\psi(\tau x)$ for $x\in E$.

Let $G$ be a $p$-adic reductive group over $F$. All representations considered are smooth complex representations. Denote by $\Rep(G)$ the category of smooth representations of $G$, by $\Irr(G)$ the set of equivalence classes of irreducible smooth representations, and by 1 or $\BC$ the trivial representation. For $\pi\in\Rep(G)$, let $\pi^\vee$ be its contragredient. For a subgroup $H\subset G$, a representation $\pi\in\Rep(G)$ is called $H$-\emph{distinguished} if $\Hom_H(\pi,\BC)\neq0$.

Let $P=MN$ be a parabolic subgroup of $G$, with unipotent radical $N$ and Levi subgroup $M$.  Denote by $\delta_P$ the modulus character of $P$, by $\Ind_P^G$  the normalized induction functor from $\Rep(M)$ to $\Rep(G)$, and by $R_P$ the normalized Jacquet functor from $\Rep(G)$ to $\Rep(M)$. The parabolic subgroup opposite to $P$ is denoted by $\overline{P}$.

For a locally profinite space $X$, let $C_c^\infty(X)$ be the space of locally constant, compactly supported $\BC$-valued functions on $X$.

\subsection{Acknowledgements}
This work was initiated after attending a talk by Wee Teck Gan at the workshop ``Automorphic Forms, Geometry and
Representation Theory'' at Zhejiang University in 2018, and was further inspired by his later talks (online resources) on Harish-Chandra characters and Weil representations. I wish to thank Wee Teck Gan for kindly sharing with me his letters \cite{gan18a, gan18b} to  Atsushi Ichino and Yiannis Sakellaridis on this topic. I have also benefited from Hengfei Lu's series of works on Prasad's conjecture. 
This work was partially supported by NSFC Grants 12022106 and 11971223.

\section{Base change doubling}\label{sec doubling}

\subsection{Base change doubled spaces}

Let $\left(V_F,\pair{\textrm{-},\textrm{-}}_{V_F}\right)$ be a quadratic space and $\left(W_F,\pair{\textrm{-},\textrm{-}}_{W_F}\right)$ a symplectic space, both over $F$.
Set $$V=V_F\otimes_FE\quad\textrm{and}\quad W=W_F\otimes_FE,$$ equipped with the $E$-bilinear forms via extension of scalars; we denote these forms by $\pair{\textrm{-} ,\textrm{-} }_V$ and $\pair{\textrm{-},\textrm{-} }_W$, respectively.  This gives rise to natural embeddings $$\O(V_F)\lhra \O(V)\quad \textrm{and}\quad \Sp(W_F)\lhra\Sp(W).$$  Fix a complete polarization $$W_F=X_F\oplus Y_F$$ of $W_F$, and set
$$X=X_F\otimes_FE\quad\textrm{and}\quad Y=Y_F\otimes_FE.$$

\begin{rem}
In this section, $V_F$ is allowed to be of arbitrary dimension, even or odd.
\end{rem}

For the symplectic space $W_F$, its \emph{base change doubled space} is defined as $$\FW:=\R_{E/F}(W).$$ Concretely, $\FW$ is the space $W$ viewed as an $F$-vector space, equipped with the symplectic form $$\pair{\textrm{-} ,\textrm{-}}_\FW=\tr_{E/F}\circ\pair{\textrm{-},\textrm{-} }_W.$$ This gives a natural embedding
$$\Sp(W)\lhra\Sp(\FW).$$

For the quadratic space $V_F$, one may define its base change doubled space $\left(\FV,\pair{\textrm{-},\textrm{-}}_\FV\right)$ in an entirely analogous manner. For our purpose, however, we require a twisted variant. 

Let $(V_\tau,\pair{\textrm{-},\textrm{-} }_{V_\tau})$ be the quadratic space over $E$ defined by taking $V_\tau=V$ as an $E$-vector space and equipping it with the bilinear form $$\pair{\textrm{-},\textrm{-}}_{V_\tau}=\tau\pair{\textrm{-},\textrm{-}}_V.$$ We have the natural identification $$\O(V_\tau)=\O(V),$$ and consequently obtain the natural embedding $$\O(V_F)\lhra\O(V)=\O(V_\tau).$$ 
 
 The $\tau$-\emph{twisted base change doubled space} is defined as   
 $$\FV_\tau:=\R_{E/F}(V_\tau),$$ equipped with the $F$-quadratic form $$\pair{\textrm{-},\textrm{-}}_{\FV_\tau}=\tr_{E/F}\circ\pair{\textrm{-},\textrm{-} }_{V_\tau}.$$ This yields a  natural embedding
 $$\O(V)=\O(V_\tau)\longhookrightarrow\O(\FV_\tau).$$
 
For brevity, we refer to $\FW$ and $\FV_\tau$ simply as the \emph{doubled spaces}.

\begin{lem}
The quadratic space $\FV_\tau$ is split.
\end{lem}

\begin{proof}
By construction, $V_F$ is an isotropic subspace of $\FV_\tau$ and satisfies $\dim V_F=\frac{1}{2}\dim \FV_\tau$. Hence $\FV_\tau$ is split.
\end{proof}

\begin{rem}
In general, the untwisted doubled space $\FV$ is not split. This is why we work with the twisted variant $\FV_\tau$.
\end{rem}

\subsection{Invariants of the double cosets} In this subsection, we
set  $$G=\Sp(W)\quad\textrm{or}\quad G=\O(V)=\O(V_\tau),$$ and accordingly $$\FG=\Sp(\FW)\quad\textrm{or}\quad \FG=\O(\FV_\tau).$$  For brevity, we write $$q=\pair{\textrm{-},\textrm{-}}_W\ \textrm{or}\ \pair{\textrm{-},\textrm{-}}_V,\quad\Fq=\pair{\textrm{-},\textrm{-}}_\FW\ \textrm{or}\ \pair{\textrm{-},\textrm{-}}_{\FV_\tau}.$$

Let $\FL$ be a fixed Lagrangian subspace of $\FW$  (resp. $\FV_\tau$), and $\FP=\Stab_\FG(\FL)$ the  Siegel parabolic subgroup that stabilizes $\FL$. 
We aim to understand the double coset space $$\FP\bs\FG/G.$$ Denote $$\CX:=\FP\bs\FG,$$ which parametrizes Lagrangians of $\FW$ (resp. $\FV_\tau$).
We fix the Lagrangian $\FL$ as follows. 
\begin{itemize}
\item For $\FW$, set \begin{equation}\label{equ lag1}\FL=X_F\oplus \tau Y_F.\end{equation}
\item For $\FV_\tau$, set \begin{equation}\label{equ lag2}\FL=\tau V_F.\end{equation}
\end{itemize}

For $L\in\CX$, define
\begin{equation}\label{equ L_an}L_\an:=\{x\in L\mid q(x,y)=0,\forall\ y\in L \}\quad\textrm{and}\quad \overline{L}:=L/L_\an.\end{equation} 
If $G=\O(V)$, since $\Fq$ vanishes on $L$,  the restriction of $q$  to $L$ takes values in $F$. Hence $q$ induces a  quadratic form 
	$$q_{\overline{L}}:\overline{L}\times \overline{L}\lra F.$$
	
\begin{prop}\label{prop orbits}
Let $L\in\CX$.
\begin{enumerate}
\item If $G=\Sp(W)$, then $\dim(\overline{L})$ is the unique invariant of the $G$-orbit of $L$ in $\CX$.
\item If $G=\O(V)$, then the isomorphism class of the quadratic space $(\overline{L},q_{\overline{L}})$ is the unique invariant of the $G$-orbit of $L$ in $\CX$.
\end{enumerate}
\end{prop}

Proposition \ref{prop orbits} is a variant of \cite[Lemma 2.1]{gpsr} and  generalizes \cite[Lemma 4.3.3]{lu20b}.

\begin{eg}
\begin{enumerate}
\item If $L=\FL$, where $\FL$ is given by (\ref{equ lag1}) for $\FW$ and by (\ref{equ lag2}) for $\FV_\tau$, then $L_\an=0$.
\item If $G=\Sp(W)$ and $L=X_F\oplus\tau X_F$, then $L_\an=L$.

\end{enumerate}
\end{eg}

In the remainder of this subsection, we assume $G=\O(V)$ and prove Proposition \ref{prop orbits} in this case. The proof for $G=\Sp(W)$ follows the same lines but is simpler.

\begin{lem}\label{lem L0}
For $L\in\CX$, we have $$L_\an=L\cap\tau L,$$ and equivalently,
$$L_\an=\{x\in V\mid q(x,y)=0,\forall\ y\in L\}.$$ In particular, $L_\an$ is an isotropic $E$-subspace of $V$.
\end{lem}

\begin{proof} We first show $L\cap\tau L\subset L_\an$.
If $x\in L\cap \tau L$, then $q(x,y)\in F$ for all $y\in L$. Write $x=\tau x_0$ with $x_0\in L$. Then $q(x,y)=\tau q(x_0,y)\in\tau F$ for all $y\in L$. Thus $q(x,y)\in F\cap\tau F=\{0\}$, so $x\in L_\an$.

Conversely, we prove $L_\an\subset L\cap\tau L$.
Since $L$ is a Lagrangian, if $x\notin L$ then there exists $y\in L$ such that $\Fq(x,y)\neq0$ and hence $q(x,y)\neq0$. This shows that $L_\an$ coincides with
	$$\{x\in V\mid q (x,y)=0,\forall\ y\in L\}.$$ This description makes it evident that $L_\an$ is an $E$-subspace of $V$. It is clear that $L\cap\tau L$ is the maximal $E$-subspace of $V$  inside $L$.  Therefore $L_\an\subset L\cap\tau L$.
\end{proof}

For any $F$-subspace $U$ of $\FV_\tau$, let $E\cdot U$ denote  the $E$-subspace of $V$ spanned by $U$; that is, $$E\cdot U=U+\tau U.$$ For each $L\in\CX$, we fix an $F$-subspace $L_\rat\subset L$ satisfying
\begin{equation}\label{equ Lra}
L=L_\an\oplus L_\rat.
\end{equation}
The subscript ``ra" stands for ``$F$-rational".

\begin{lem}\label{lem lag}
Let $L\in\CX$. Then the following hold.
\begin{enumerate}
\item As quadratic spaces over $F$, we have
$(\overline{L},q_{\overline{L}})\cong(L_\rat,q|_{L_\rat})$.
\item The subspace  $E\cdot L_\rat$ is a quadratic $E$-subspace of $V$, and there is an isomorphism $E\cdot L_\rat\cong\overline{L}\otimes_FE$ of quadratic spaces over $E$.
\item The subspace $L_\an$ is a Lagrangian of the orthogonal complement  $(E\cdot L_\rat)^\bot$ in $V$.
\end{enumerate}
\end{lem}

\begin{proof}
(1) The first assertion is immediate from the definitions of $\overline{L}$ and $L_{\rat}$.

(2) By Lemma \ref{lem L0}, it is clear that
$$E\cdot L=L_\an\oplus E\cdot L_\rat\quad\textrm{and}\quad E\cdot L_\rat\cong L_\rat\otimes_FE.$$ Then the second assertion follows from the first one.

(3) Set $k=\dim_EL_\an$. From the short exact sequence 
$$0\lra L_\an\lra L\oplus\tau L\stackrel{ (x,y)\mapsto x-y}{\lra} E\cdot L\lra0,$$ we see that $$\dim_E(E\cdot L)=\dim_EV-k,$$ and $$\dim_E(E\cdot L_\rat)=\dim_EV-2k.$$ 
 Since $L_\an$ is isotropic and $\dim_EL_\an=k=\frac{1}{2}\dim_E(E\cdot L_\rat)^\bot$, the subspace $L_\an$ is a Lagrangian of $(E\cdot L_\rat)^\bot$.
\end{proof}

\begin{proof}[Proof of Proposition \ref{prop orbits}]
Let $L_1, L_2\in\CX$. 

If $gL_1=L_2$ for some $g\in G$, then clearly $gL_{1,\an}=L_{2,\an}$. Consequently, $(\overline{L}_1,q_{\overline{L}_1})\simeq(\overline{ L}_2,q_{\overline{L}_2})$. 

Suppose $(\overline{L}_1,q_{\overline{L}_1})\simeq(\overline{ L}_2,q_{\overline{L}_2})$. 
By Lemma \ref{lem lag}, the quadratic spaces $L_{1,\rat}$ and 
$ L_{2,\rat}$ are isomorphic over $F$. Choose an $F$-isomertry $g_\rat: L_{1,\rat}\ra L_{2,\rat},$ and extend it $E$-linearly to an $E$-isometry $$g_\rat: E\cdot L_{1,\rat}\lra E\cdot L_{2,\rat}.$$ Lemma \ref{lem lag} further implies that there is an $E$-isometry $$g_\an:(E\cdot L_{1,\rat})^\bot\lra(E\cdot L_{2,\rat})^\bot$$ such that $g_\an L_{1,\an}=L_{2,\an}$. Let $g=g_\an\oplus g_\rat$. Then $g\in G$ and $gL_1=L_2$. This completes the proof.
\end{proof}

\subsection{Open orbits and their stabilizers} The following characterization is standard.

\begin{lem}\label{lem open-orbit}
For $L\in\CX$, the $G$-orbit of $L$ is open in $\CX$ if and only if $L_\an=\{0\}$.
\end{lem}

Throughout this subsection we assume $L_\an=\{0\}$ and analyze the stabilizer $\Stab_G(L)$. Orbits of this type are referred to as {\em main orbits} in \cite{gpsr}.

When $L=\FL$ as in (\ref{equ lag1}) and (\ref{equ lag2}), we have $L_\an=\{0\}$ and $$\Stab_G(\FL)=\FP\cap G.$$ If $G=\Sp(W)$, Proposition \ref{prop orbits} shows that the $G$-orbit of $\FL$ is the unique orbit  satisfying $L_\an=\{0\}$. 

For $G=\Sp(W)$, we fix an element of the similitude group \begin{equation}\label{equ similitude}c_\tau\in\GSp(W)\end{equation} that acts on $X$ as scalar multiplication by $\tau$ and on $Y$ as the identity.

The next lemma is immediate.

\begin{lem}
\begin{enumerate}
\item
If $G=\Sp(W)$, then $\Stab(\FL)=\Sp(W_F)^\tau$, where $\Sp(W_F)^\tau:=c_\tau\Sp(W_F)c_\tau^{-1}$.
\item If $G=\O(V)$, then $\Stab(\FL)=\O(V_F)$.
\end{enumerate}
\end{lem}

In the remainder of this subsection we assume $G=\O(V)$.

\begin{lem}
If $L_\an=\{0\}$, then $L$ is a quadratic $F$-space and, as quadratic spaces over $E$, $$L\otimes_FE\cong V.$$  In this situation, $\Stab_G(L)=\O(L)$. 
\end{lem}
\begin{proof}
The condition $L_{\an}=\{0\}$ implies $V=E\cdot L$, which yields the isomorphism $L\otimes_FE\cong V$  by Lemma \ref{lem lag}. The natural inclusion $\O(L)\incl \O(V)=G$ shows that $\O(L)\subset\Stab_G(L)$. 
Conversely, any $g\in\Stab_G(L)$ is an isometry of $V$ and preserves $L$. Therefore $g\in\O(L)$. 
Hence $\Stab_G(L)=\O(L)$. 
\end{proof}

Let $\RH^1(E/F,\O(V))$ be the first Galois cohomology set of $\O(V)$ whose $F$-rational structure is induced from $\O(V_F)$ and on which $\Gal(E/F)$ acts naturally. Explicitly, $$\RH^1(E/F,\O(V))=\left\{g\in\O(V)\mid g\iota(g)=1\right\}/\sim,$$ where $g\sim g'$ if and only if there exists $h\in\O(V)$ such that $g'=hg\iota(h)^{-1}$. This set classifies  isomorphism classes of quadratic spaces $V'_F$ over $F$ whose base change $V'=V_F'\otimes_FE$ is isomorphic to $V$.

\begin{lem}\label{lem orbits-coho1}
The $G$-orbits of $L$ with $L_\an=\{0\}$ are in bijection with  
$\RH^1(E/F,\O(V))$.
\end{lem}

\begin{proof}
The explicit parametrization is given in the following way. Let a cohomology class be represented by a cocycle  $g\in\O(V)$ with $g\iota(g)=1$. Set $$L^g:=\left\{v\in V\mid g\cdot\iota(v)=v\right\}.$$ It is straightforward to verify that $L^g\in\CX$ and $L^g_\an=\{0\}$. If $g'=hg\iota(h)^{-1}$ for some $h\in\O(V)$, then $L^{g'}=hL^g$. Thus the correspondence $g\mapsto L^g$ descends to a well-defined map from $\RH^1(E/F,\O(V))$ to the set of $G$-orbits of such Lagrangians. 

To see surjectivity, take any $L\in\CX$ with $L_\an=\{0\}$. By Lemma \ref{lem lag}, there exists an $E$-isometry $\phi:L\otimes_FE\ra V$. Define $g=\phi\circ\iota(\phi)^{-1}\in\O(V)$. Then $g\iota(g)=1$ and the associated Lagrangian $L^g$ equals $L$. It is routine to check the injectivity. Hence we obtain the desired bijection.
\end{proof}

\subsection{Negligible orbits and their stabilizers} Throughout this subsection we assume $L_\an\neq\{0\}$.  Our goal is to show that $L$ is {\em negligible} in the terminology of \cite{gpsr}; that is, $\Stab_G(L)$ contains the unipotent radical $N$ of a proper parabolic subgroup of $G$ as a normal subgroup. 

Denote $$R:=\Stab_G(L)\quad\textrm{and}\quad P:=\Stab_G(L_\an).$$ Lemma \ref{lem lag} implies that $P$ is a proper parabolic subgroup of $G$. Let $N$ be the unipotent radical of $P$.

\begin{lem}
If $L_\an\neq\{0\}$, then $N$ is a normal subgroup of $R$.
\end{lem}

\begin{proof} We prove the lemma for $G=\O(V)$; the case $G=\Sp(W)$ is similar and omitted.

It is clear that $R\subset P$. Thus it suffices to show $N\subset R$. The unipotent radical $N$ is given explicitly by
$$N=\left\{g\in G\mid g|_{L_\an}=\id,\ g|_{L_\an^{\bot}/L_\an}=\id \right\}.$$ 
Take $g\in N$. Since $L\in\CX$ is a Lagrangian, to prove $g\in R$ it is enough to verify $\Fq(gL,L)=0$,  i.e. $q(gL,L)\subset F.$ By Lemma \ref{lem lag}, $L_\an^{\bot}=E\cdot L$ and $L/L_\an\subset L_\an^{\bot}/L_\an.$ Because $g$ acts trivially on $L_\an^\bot/L_\an$, it also acts trivially on the subspace $L/L_\an$. Hence for any $x\in L$ we can write $gx=x+a$ with $a\in L_\an$. For arbitrary $x,y\in L$, the fact that $a\in L_\an$ implies $q(a,y)=0$. Thus
$$q(gx,y)=q(x+a,y)=q(x,y)\in F.$$ This establishes $N\subset R$.
\end{proof}

\subsection{Explicit orbits and stabilizers}\label{subsec explicit orbit}
Let $r$ be the split rank of $G$ (equivalently, the Witt index of  $W$ or $V$).  For $0\leq k\leq r$, set
\begin{equation}\label{equ X_k}\CX_k:=\left\{L\in\CX\mid\dim_EL_\an=k\right\}.\end{equation} Each $\CX_k$ is $G$-stable, and $\CX$ admits the decomposition into  disjoint $G$-orbits
\begin{equation}\label{equ disj}\CX=\bigsqcup_{k=0}^{r}\CX_k.\end{equation}

\subsubsection{The symplectic case} Assume $G=\Sp(W)$ and $\dim(W)=2n$. By Proposition \ref{prop orbits}, each $\CX_k$ is a single $G$-orbit.  We now construct an explicit representative $L_k\in\CX_k$ and determine its stabilizer $R=\Stab_G(L_k)$.

Write $$X_F=Fe_1\oplus\cdots\oplus Fe_n,\quad Y_F=Fe'_1\oplus\cdots\oplus Fe'_n,$$ where $\{e'_1,...,e'_n\}$ is the basis dual to $\{e_1,...,e_n\}$ under the symplectic form $q$. Define the following subspaces:
$$\begin{aligned}&X_{F,k}:=Fe_1\oplus\cdots\oplus Fe_k,\quad X_k:=X_{F,k}\otimes_FE,\\ 
&X'_{F,n-k}:=Fe_{k+1}\oplus\cdots\oplus Fe_n,\quad Y'_{F,n-k}:=Fe'_{k+1}\oplus\cdots\oplus Fe'_n,\\ 
&W_{2n-2k,F}:=X'_{F,n-k}\oplus Y'_{F,n-k},\quad W_{2n-2k}:=W_{F,2n-2k}\otimes_FE.\end{aligned}$$
Now set \begin{equation}\label{equ lag-k}L_k:=X_k\oplus X'_{F,n-k}\oplus \tau Y'_{F,n-k}.\end{equation} A direct verification shows $L_k\in\CX$ and $L_{k,\an}=X_k$, so indeed $L_k\in\CX_k$. In terms of the above decomposition, the parabolic $P=\Stab_G(L_{k,\an})$ is
$$P=\left(\GL(X_k)\times\Sp(W_{2n-2k})\right)\ltimes N.$$

\begin{lem}\label{lem stab_Sp}
For $L_k$ as in (\ref{equ lag-k}), its stablizer is
$$R=\left(\GL(X_k)\times\Sp(W_{2n-2k,F})^\tau\right)\ltimes N.$$ 
\end{lem}

\begin{proof}
The proof is identical to that of Lemma \ref{lem stab_O} below and is therefore omitted.
\end{proof}

\subsubsection{The orthogonal case} Now assume $G=\O(V)$ and $\dim V=m$. Let $r$ be the Witt index of $V$. Choose a maximal isotropic $E$-subspace $X_r\subset V$ with basis $e_1,...,e_{r}$, and let $e'_1,...,e'_r$ be the dual basis with respect to the quadratic form $q$. Set  $Y_r=Ee_1'\oplus\cdots\oplus Ee_r'$.
There exists $L\in\CX$ such that $X_{r}\subset L$. It is clear that $L_\an=X_r$ and $E\cdot L_\rat$ is anisotropic. 
 For $0\leq k\leq r$, define
$$X_k:=Ee_1\oplus\cdots\oplus Ee_k$$ and 
$$V_{m-2k,F}:=\left(Fe_{k+1}\oplus\cdots\oplus Fe_{r}\right)\bigoplus L_\rat\bigoplus\left(Fe'_{k+1}\oplus\cdots\oplus Fe'_{r}\right).$$ Now set
\begin{equation}\label{equ lag-k2}L_k:=X_k\oplus \tau V_{m-2k,F}.\end{equation}
A direct computation shows $$L_{k,\an}=X_k,\quad L_{k,\rat}=\tau V_{m-2k,F},$$ and therefore $L_k\in\CX_k$. The parabolic $P=\Stab_G(L_{k,\an})$ is
$$P=\left(\GL(X_k)\times\O(V_{m-2k})\right)\ltimes N,$$
where $V_{m-2k}:=V_{m-2k,F}\otimes_FE$.

\begin{lem}\label{lem stab_O}
For $L_k$ as in (\ref{equ lag-k2}), its stabilizer is 
$$R=\left(\GL(X_k)\times\O(V_{m-2k,F})\right)\ltimes N.$$
\end{lem}

\begin{proof}
Since $R\subset P=MN$, where $M=\GL(X_k)\times\O(V_{m-2k})$, and we have already shown $N\subset R$, it follows that
$$R=(R\cap M)\ltimes N.$$ Take $m=(m_\an,m_\rat)\in M$ with $m_\an\in\GL(X_k)$ and $m_\rat\in\O(V_{m-2k})$. Then $$m L_k=m_\an X_k\oplus m_\rat L_{k,\rat}=X_k\oplus m_\rat\cdot \tau V_{m-2k,F}.$$ Hence $mL_k=L_k$ if and only if $m_\rat V_{m-2k,F}=V_{m-2k,F}$, which is equivalent to $m_\rat\in\O(V_{m-2k,F})$. This completes the proof.
\end{proof}

The following lemma is a generalization of Lemma \ref{lem orbits-coho1}.

\begin{lem}\label{lem orbits-coho2}
For each $k$, the $G$-orbits in $\CX_k$ are parametrized by $\RH^1(E/F,\O(V_{m-2k}))$, where the $F$-structure on $\O(V_{m-2k})$ is inherited from $\O(V_{m-2k,F})$.
\end{lem}

\section{Weil representations}\label{sec weil}

 \subsection{Dual pairs} From now on we fix
 $$\dim V_F=2m,\quad\dim W_F=2n.$$
Let $\BW:=V\otimes_EW$, equipped with the symplectic form $$\pair{\textrm{-},\textrm{-} }_\BW=\pair{\textrm{-},\textrm{-}}_V\otimes\pair{\textrm{-},\textrm{-} }_W.$$ Then $\left(\O(V),\Sp(W)\right)$ is a reductive dual pair inside $\Sp(\BW)$. 
Set $\BW^\square:=\R_{E/F}(\BW)$ and define its symplectic form by $$\pair{\textrm{-},\textrm{-}}_{\BW^\square}:=\tr_{E/F}\circ\pair{\textrm{-},\textrm{-}}_\BW.$$ This yields a natural embedding $$\Sp(\BW)\lhra\Sp(\BW^\square).$$  
The space $\BW^\square$ can be canonically identified in two ways: $$\BW^\square=\FV\otimes_F W_F=V_F\otimes_F\FW,$$ and under these identifications the symplectic form becomes
$$\pair{\textrm{-},\textrm{-}}_{\BW^\square}=\pair{\textrm{-},\textrm{-}}_\FV\otimes\pair{\textrm{-},\textrm{-} }_{W_F}=\pair{\textrm{-},\textrm{-} }_{V_F}\otimes\pair{\textrm{-},\textrm{-}}_\FW.$$
Hence $(\O(\FV),\Sp(W_F))$ and $(\O(V_F),\Sp(\FW))$ are dual pairs in $\Sp(\BW^\square)$. 

The dual pairs $\left(\O(V),\Sp(W)\right)$ and $(\O(V_F),\Sp(\FW))$ form \emph{base change seesaw pairs} (cf. \cite[Sect. 15.1.4]{gkt}) in $\Sp(\BW^\square)$. The corresponding seesaw diagram is
\begin{equation}\label{equ seesaw pair1}\begin{tikzcd}
\O(V) \arrow[d, no head] \arrow[rd, no head] & \Sp(\FW) \arrow[d, no head] \\
\O(V_F) \arrow[ru, no head]                      & \Sp(W).                
\end{tikzcd}\end{equation}

Similarly, we define the symplectic spaces
$$\BW_\tau=V_\tau\otimes_EW,\quad \BW^\square_\tau=\R_{E/F}(\BW_\tau).$$ Then $\left(\O(V_\tau),\Sp(W)\right)$ and $\left(\O(\FV_\tau),\Sp(W_F)\right)$ are base change seesaw pairs in $\Sp(\BW_\tau)$, with seesaw diagram 
\begin{equation}\label{equ seesaw pair2}\begin{tikzcd}
\O(\FV_\tau) \arrow[d, no head] \arrow[rd, no head] & \Sp(W) \arrow[d, no head] \\
\O(V_\tau) \arrow[ru, no head]                      & \Sp(W_F).                 
\end{tikzcd}\end{equation}

\subsection{Metaplectic groups and splittings}

\subsubsection{}\label{subsec splitting} We denote by $\Mp(\BW)$ the metaplectic 
$\BC^1$-cover of $\Sp(\BW)$, and by $\Mp(\BW^\square)$ the metaplectic $\BC^1$-cover of $\Sp(\BW^\square)$ (cf. \cite[Sect. 2.1]{gkt}); their explicit realizations (via the Schrodinger models) will be given below.
Here $\BC^1:=\{z\in\BC\mid |z|=1\}$.

Recall that  $W_F=X_F\oplus Y_F$ is a complete polarization of the symplectic space $W_F$. We extend this polarization to the various symplectic spaces as follows:
$$\begin{aligned}
&X:=X_F\otimes_FE,\quad Y:=Y_F\otimes_FE, \\
&\FX:=\R_{E/F}(X),\quad \FY:=\R_{E/F}(Y),\\ 
&\BX:=X\otimes_EV,\quad \BY:=Y\otimes_EV,
\end{aligned}$$ and
$$\begin{aligned}&\BX^\square:=\R_{E/F}(\BX)=X_F\otimes_F\FV=\FX\otimes_F V_F,\\
&\BY^\square:=\R_{E/F}(\BY)=Y_F\otimes_F\FV=\FY\otimes_FV_F.\end{aligned}$$
 Then $$\begin{aligned}&W=X\oplus Y,\quad \FW=\FX\oplus\FY,\\ &\BW=\BX\oplus\BY,\quad \BW^\square=\BX^\square\oplus\BY^\square\end{aligned}$$ are complete polarizations of $W, \FW, \BW$ and $\BW^\square$, respectively.
 
We use the Leray cocycle $c_\BX^\psi$ (cf. \cite[Definition 2.7]{gkt}) to realize the metaplectic cover $\Mp(\BW)$ as the set  $\Sp(\BW)\times\BC^1$. The group law on $\Mp(\BW)=\Sp(\BW)\times\BC^1$ is given by
$$(g_1,z_1)\cdot(g_2,z_2)=\left(g_1g_2,c_\BX^\psi(g_1,g_2)z_1z_2\right).$$
In the same way, we use the Leray cocycle
$c_{\BX^\square}^{\psi_F}$ to realize $\Mp(\BW^\square)$ as $\Sp(\BW^\square)\times\BC^1$. These two cocycles satisfy a key compatibility property (cf. \cite[Proposition 6.26]{gkt}): for every $g_1,g_2\in\Sp(\BW)$,  
$$c_\BX^\psi(g_1,g_2)=c_{\BX^\square}^{\psi_F}(g_1,g_2).$$  This compatibility yields the \emph{base change embedding} (cf. \cite[Definition 6.27]{gkt}) \begin{equation}\label{equ bc-map1}\bc:\Mp(\BW)\lhra\Mp(\BW^\square),\quad (g,z)\mapsto(g,z),\end{equation} which is the unique lift of the natural embedding $\Sp(\BW)\incl\Sp(\BW^\square)$.

Associated with the complete polarizations fixed above and the additive characters $\psi_F, \psi$, we obtain two splittings (cf. \cite[Corollary 11.7]{gkt}): 
$$s_{\psi,X}^V:\Sp(W)\lra\Mp(\BW)\stackrel{\bc}{\lhra}\Mp(\BW^\square),\quad s_{\psi_F,\FX}^{V_F}:\Sp(\FW)\lra\Mp(\BW^\square).$$ Since the splitting is unique for a given symplectic group, the two splittings are compatible;  that is,  $$s_{\psi_F,\FX}^{V_F}|_{\Sp(W)}=s_{\psi,X}^V.$$ 

For the orthogonal groups we use the canonical splittings
$$
s_{\psi,X}^W:\O(V)\lra\Mp(\BW)\stackrel{\bc}{\lhra}\Mp(\BW^\square),\quad
h\mapsto (h\otimes 1_W,1),$$ and
$$s_{\psi_F,\FX}^\FW:\O(V_F)\lra\Mp(\BW^\square),\quad
h\mapsto (h\otimes 1_\FW,1).$$
Compatibility of these two splittings is immediate from their definitions.

Thus we obtain compatible splittings for the base change seesaw pair \eqref{equ seesaw pair1}, summarized by the following commutative diagram:
\begin{equation}\label{equ comp-seesaw1}
\begin{tikzcd}
\Sp(\FW)                 & \times & \O(V_F) \arrow[d, hook] \arrow[rrr, "{\left(s_{\psi_F,\FX}^{V_F},\ s_{\psi_F,\FX}^\FW\right)}"] &  &                    & \Mp(\BW^\square) \arrow[d, no head, equal] \\
\Sp(W) \arrow[u, hook]   & \times & \O(V) \arrow[rr, "{\left(s_{\psi,X}^{V},\ s_{\psi,X}^W\right)}"]    &  & \Mp(\BW) \arrow[r, hook, "{\bc}"] & \Mp(\BW^\square).
\end{tikzcd}\end{equation}

\subsubsection{} Now consider the analogous splittings for the base change seesaw pairs (\ref{equ seesaw pair2}). 

Define
$$\begin{aligned}&\BX_\tau:=X\otimes_E V_\tau,\quad\BY_\tau:=Y\otimes_E V_\tau,\\
&\BX^\square_\tau:=\R_{E/F}(\BX_\tau)=X_F\otimes_F\FV_\tau,\\&\BY^\square_\tau:=\R_{E/F}(\BY_\tau)=Y_F\otimes_F\FV_\tau.
\end{aligned}$$  Then $$\BW_\tau=\BX_\tau\oplus\BY_\tau,\quad \BW_\tau^\square=\BX_\tau^\square\oplus\BY_\tau^\square$$ are complete polarizations of $\BW_\tau$ and $\BW_\tau^\square$,  respectively.

There are natural identifications
$$\BX_\tau=\BX,\quad  \BY_\tau=\BY,\quad \BX_\tau^\square=\BX^\square,\quad \BY_\tau^\square=\BY^\square,$$ induced by the natural identification $V=V_\tau$.

We explicate $\Mp(\BW_\tau)$ (resp. $\Mp(\BW^\square_\tau)$) as  
$\Sp(\BW_\tau)\times\BC^1$ (resp. $\Sp(\BW^\square_\tau)\times\BC^1$) via the Leray cocycle $c^\psi_{\BX_\tau}$ (resp. $c^{\psi_F}_{\BX_\tau^\square}$). As in (\ref{equ bc-map1}), we have the base change embedding \begin{equation}\label{equ bc-map2}\bc:\Mp(\BW_\tau)\lhra\Mp(\BW_\tau^\square),\quad (g,z)\mapsto(g,z).\end{equation}
For the symplectic groups we have the compatible splittings
$$s_{\psi,X}^{V_\tau}:\Sp(W)\lra\Mp(\BW_\tau)\stackrel{\bc}{\lhra}\Mp(\BW^\square_\tau),\quad s_{\psi_F,X_F}^{\FV_\tau}:\Sp(W_F)\lra\Mp(\BW^\square_\tau),$$ and for the orthogonal groups the compatible canonical splittings
$$
s_{\psi,X}^W:\O(V)\lra\Mp(\BW_\tau)\stackrel{\bc}{\lra}\Mp(\BW_\tau^\square),\quad s_{\psi_F,X_F}^{W_F}:\O(\FV_\tau)\lra\Mp(\BW_\tau^\square).$$

In summary, the compatible splittings for the base change seesaw pairs (\ref{equ seesaw pair2}) are encoded in the following commutative diagram:
\begin{equation}\label{equ comp-seesaw2}
\begin{tikzcd}
\Sp(W_F) \arrow[d, hook]                 & \times & \O(\FV_\tau)  \arrow[rrr, "{\left(s_{\psi_F,X_F}^{\FV_\tau},\ s_{\psi_F,X_F}^{W_F}\right)}"] &  &                    & \Mp(\BW^\square) \arrow[d, no head, equal] \\
\Sp(W)   & \times & \O(V_\tau) \arrow[u, hook] \arrow[rr, "{\left(s_{\psi,X}^{V_\tau},\ s_{\psi,X}^W\right)}"]    &  & \Mp(\BW) \arrow[r, hook, "{\bc}"] & \Mp(\BW^\square).
\end{tikzcd}\end{equation}

\subsubsection{} We now recall the relation between $\Mp(\BW)$ and $\Mp(\BW_\tau)$, as well as the relation between the splittings of the dual pairs $(\O(V),\Sp(W))$ and $(\O(V_\tau),\Sp(W))$.

Let $\Mp(\BW)_{\psi_\tau}=\Sp(\BW)\times\BC^1$ be the realization of the metaplectic group defined by the Leray cocycle $c_\BX^{\psi_\tau}$. The natural identification  $\Sp(\BW_\tau)=\Sp(\BW)$ lifts uniquely to an isomorphism of metaplectic groups
\begin{equation}\label{equ meta-id}
\Mp(\BW)_{\psi_\tau}\stackrel{\simeq}{\lra}\Mp(\BW_\tau),\quad
(g,z)\mapsto(g,z).
\end{equation} Furthermore, the splittings are compatible in the sense of the following commutative diagram  (cf. \cite[Proposition 12.10]{gkt})
\begin{equation}\label{equ split-id1}\begin{tikzcd}
\Sp(W) \arrow[rr,"s_{\psi_\tau, X}^V"] \arrow[d, no head, equal] &  & \Mp(\BW)_{\psi_\tau}=\Sp(\BW)\times\BC^1 \arrow[d, no head, equal] \\
\Sp(W) \arrow[rr,"s_{\psi,X}^{V_\tau}"]                                &  & \Mp(\BW_\tau)=\Sp(\BW_\tau)\times\BC^1,                     
\end{tikzcd}\end{equation}
where $s_{\psi_\tau,X}^V$ is the splitting associated to the additive character $\psi_\tau$ and the isotropic subspace $X$.

Let  $$s_{\psi_\tau,X}^{W}:\O(V)\lra\Mp(\BW)_{\psi_\tau},\quad
h\mapsto(h\otimes 1_W,1),$$ be the canonical splitting.
Under the natural identification  $\O(V_\tau)=\O(V)$, 
the following diagram commutes:
\begin{equation}\label{equ split-id2}\begin{tikzcd}
\O(V) \arrow[rr,"s_{\psi_\tau, X}^W"] \arrow[d, no head, equal] &  & \Mp(\BW)_{\psi_\tau}=\Sp(\BW)\times\BC^1 \arrow[d, no head, equal] \\
\O(V_\tau) \arrow[rr,,"s_{\psi,X}^W"]                                &  & \Mp(\BW_\tau)=\Sp(\BW_\tau)\times\BC^1.                    
\end{tikzcd}\end{equation}

\subsection{Weil representations}
\subsubsection{}
Let $\omega_{\psi,\BW}$ (resp. $\omega_{\psi_F,\BW^\square}$) be the Weil representation of $\Mp(\BW)$ (resp. $\Mp(\BW^\square)$) associated to the additive character $\psi$ (resp. $\psi_F$). These representations are realized on the Schrodinger models with respect to the complete polarizations fixed in Sect. \ref{subsec splitting}:  
$$\omega_{\psi,\BW}\ \ \textrm{on}\ \ C_c^\infty(\BY),\qquad \omega_{\psi_F,\BW^\square}\ \ \textrm{on}\ \ C_c^\infty(\BY^\square).$$ The natural identification $\BY^\square=\BY$  induces an identification $C_c^\infty(\BY^\square)=C_c^\infty(\BY)$. Under this  identification,   we have (cf. \cite[Proposition 9.21]{gkt})
\begin{equation}\label{equ res-weil 1}
\omega_{\psi_F,\BW^\square}|_{\Mp(\BW)}=\omega_{\psi,\BW}.
\end{equation}
Now, via the compatible splittings of the seesaw pair (\ref{equ comp-seesaw1}), let
$$\omega_{V,W,\psi}:=\omega_{\psi,\BW}|_{\Sp(W)\times\O(V)},\quad
\omega_{V_F,\FW,\psi_F}:=\omega_{\psi_F,\BW^\square}|_{\Sp(\FW)\times\O(V_F)}.$$
Then we obtain the following compatibility relations:
\begin{equation}\label{equ res-weil 2}
\omega_{V,W,\psi}|_{\O(V_F)}=\omega_{V_F,\FW,\psi_F}|_{\O(V_F)},\quad\omega_{V,W,\psi}|_{\Sp(W)}=\omega_{V_F,\FW,\psi_F}|_{\Sp(W)}.
\end{equation}

\subsubsection{} Similarly, let $\omega_{\psi,\BW_\tau}$ (resp. $\omega_{\psi_F,\BW_\tau^\square}$) be the Weil representation of $\Mp(\BW_\tau)$ (resp. $\Mp(\BW^\square_\tau)$),  both realized on the Schrodinger model $C_c^\infty(\BY_\tau)=C_c^\infty(\BY_\tau^\square)$. Using the compatible splittings of the seesaw pair \eqref{equ comp-seesaw2}, let
$$\omega_{V_\tau,W,\psi}:=\omega_{\psi,\BW_\tau}|_{\Sp(W)\times\O(V_\tau)},\quad
\omega_{\FV_\tau,W_F,\psi_F}:=\omega_{\psi_F,\BW_\tau^\square}|_{\Sp(W_F)\times\O(\FV_\tau)}.$$ Then we obtain the following compatibility relations:
\begin{equation}\label{equ res-weil 3}
\omega_{V_\tau,W,\psi}|_{\Sp(W_F)}=\omega_{\FV_\tau,W_F,\psi_F}|_{\Sp(W_F)},\quad \omega_{V_\tau,W,\psi}|_{\O(V_\tau)}=\omega_{\FV_\tau,W_F,\psi_F}|_{\O(V_\tau)}.
\end{equation}

\subsubsection{}
Recall that $\Mp(\BW)_{\psi_\tau}$ denotes the concrete realization 
$\Sp(\BW)\times\BC^1$ defined by the Leray cocycle $c_\BX^{\psi_\tau}$. Let $\omega_{\psi_\tau,\BW}$ be the Weil representation of $\Mp(\BW)_{\psi_\tau}$, realized on the Schrodinger model $C_c^\infty(\BY)$. Under the natural identification  $C_c^\infty(\BY)=C_c^\infty(\BY_\tau)$ and the identification $\Mp(\BW)_{\psi_\tau}=\Mp(\BW_\tau)$ via (\ref{equ meta-id}), we have \begin{equation}\label{equ weil equ1}\omega_{\psi_\tau,\BW}=\omega_{\psi,\BW_\tau}.\end{equation}
Let $$\omega_{V,W,\psi_\tau}:=\omega_{\psi_\tau,\BW}|_{\Sp(W)\times\O(V)},$$ using the splitting $(s^{V}_{\psi_\tau,X},s^W_{\psi_\tau,X})$. 
Then, by (\ref{equ weil equ1}), (\ref{equ split-id1}) and (\ref{equ split-id2}), we obtain
\begin{equation}\label{equ weil equ2}\omega_{V,W,\psi_\tau}=\omega_{V_\tau, W,\psi}.\end{equation}
In particular, together with (\ref{equ res-weil 3}) we have
\begin{equation}\label{equ weil equ3}\omega_{\FV_\tau,W_F,\psi_F}|_{\O(V_\tau)}=\omega_{V_\tau,W,\psi}|_{\O(V_\tau)}=\omega_{V,W,\psi_\tau}|_{\O(V)}.\end{equation}

Recall the element  $c_\tau\in\GSp(W)$ defined in (\ref{equ similitude}). Let $$\varphi_\tau:\Sp(W)\lra\Sp(W),\quad g\mapsto c_\tau gc_\tau^{-1}.$$
For a representation $\Pi\in\Rep(\Sp(W))$, we set 
$$\Pi_\tau:=\Pi\circ \varphi_\tau^{-1},\quad\Pi^\tau:=\Pi\circ \varphi_\tau.$$
A fundamental compatibility (cf. \cite[Proposition 9.18]{gkt}) is
\begin{equation}
\omega_{V,W,\psi_\tau}|_{\Sp(W)}=\omega_{V,W,\psi}|_{\Sp(W)}\circ\varphi_\tau.
\end{equation}
Together with (\ref{equ res-weil 2}) and (\ref{equ weil equ2}), this yields 
\begin{equation}\label{equ weil equ4}
\omega_{V_F,\FW,\psi_F}|_{\Sp(W)}\circ\varphi_\tau=\omega_{V,W,\psi}|_{\Sp(W)}\circ\varphi_\tau=\omega_{V,W,\psi_\tau}|_{\Sp(W)}=\omega_{V_\tau,W,\psi}|_{\Sp(W)}.
\end{equation}

\subsection{Transitions of models} Other models of the Weil representations will be needed later; the transitions between these models are essential for our work.

Consider the Weil representation $\omega_{V_F,\FW,\psi_F}$. We introduce an alternative complete polarization $$\FW=\FX'\oplus\FY',$$ where $$\FX':=X_F\oplus \tau Y_F,\quad \FY':=Y_F\oplus\tau X_F.$$ This gives rise to the complete polarization
\begin{equation}\label{equ polar-1}\BW^\square=\BX^{\square'}\oplus\BY^{\square'},\end{equation} with
$$\begin{aligned}&\BX^{\square'}:=\FX'\otimes_FV_F=(X_F\otimes V_F)\bigoplus (\tau Y_F\otimes V_F),\\ 
&\BY^{\square'}:=\FY'\otimes_FV_F=(Y_F\otimes V_F)\bigoplus (\tau X_F\otimes V_F).\end{aligned}$$
For comparison, the original polarization is $\BW^\square=\BX^{\square}\oplus \BY^{\square}$, with
$$\begin{aligned}&\BX^{\square}=\FX\otimes_FV_F=(X_F\otimes V_F)\bigoplus (\tau X_F\otimes V_F),\\ 
&\BY^{\square}=\FY\otimes_FV_F=(Y_F\otimes V_F)\bigoplus(\tau Y_F\otimes V_F).\end{aligned}$$
If we now realize $\omega_{V_F,\FW,\psi_F}$ on the Schrodinger model $C_c^\infty(\BY^{\square'})$ with respect to the  alternative polarization (\ref{equ polar-1}),
the intertwining operator $$\CI_\FW:C_c^\infty(\BY^\square)\lra C_c^\infty(\BY^{\square'})$$ between the two Schrodinger models is given by (cf. \cite[Sect. 9.2.3]{gkt})
\begin{equation}\label{equ int-1}\CI_\FW(\phi)(x+y)=\int_{ \tau Y_F\otimes V_F}\phi(y+y')\psi_F\left(\pair{y',x}_{\BW^\square} \right)\ \d y',\end{equation}
where $x\in \tau X_F\otimes V_F$ and $y\in Y_F\otimes V_F.$
In particular, evaluating at zero yields
\begin{equation}\label{equ int-val1}\CI_\FW(\phi)(0)=\int_{\tau Y_F\otimes V_F}\phi(y')\ \d y'.\end{equation}
 
For the Weil representation $\omega_{\FV_\tau,W_F,\psi_F}$,  the complete polarization $$\FV_\tau=\tau V_F\oplus V_F$$ yields an alternative complete polarization of $\BW_\tau^\square$:
\begin{equation}\label{equ polar-2}\BW_\tau^\square=\BX_\tau^{\square'}\oplus\BY_\tau^{\square'},\end{equation}  where
$$\begin{aligned}&\BX_\tau^{\square'}:=W_F\otimes \tau V_F=(X_F\otimes \tau V_F)\bigoplus (Y_F\otimes \tau V_F),\\ 
&\BY^{\square'}_\tau:=W_F\otimes V_F=(X_F\otimes V_F)\bigoplus  (Y_F\otimes V_F).\end{aligned}$$
For comparison, the original polarization is $\BW_\tau^\square=\BX_\tau^{\square}\oplus \BY_\tau^{\square}$, with
$$\begin{aligned}&\BX_\tau^{\square}=X_F\otimes_F\FV_\tau=(X_F\otimes V_F)\bigoplus (X_F\otimes\tau V_F),\\ 
&\BY_\tau^{\square}=Y_F\otimes_F\FV_\tau=(Y_F\otimes V_F)\bigoplus(Y_F\otimes \tau V_F).\end{aligned}$$
If we now realize $\omega_{\FV_\tau,W_F,\psi_F}$ on the Schrodinger model $C_c^\infty(\BY_\tau^{\square'})$ with respect to the alternative polarization (\ref{equ polar-2}),
the intertwining operator $$\CI_{\FV_\tau}:C_c^\infty(\BY_\tau^\square)\lra C_c^\infty(\BY_\tau^{\square'})$$ between the two Schrodinger models is given by 
\begin{equation}\label{equ int-2}\CI_{\FV_\tau}(\phi)(x+y)=\int_{Y_F\otimes \tau V_F}\phi(y+y')\psi_F\left(\pair{y',x}_{\BW_\tau^\square} \right)\ \d y',\end{equation}
where $x\in X_F\otimes V_F$ and $ y\in Y_F\otimes V_F.$
In particular, we have
\begin{equation}\label{equ int-val2}\CI_{\FV_\tau}(\phi)(0)=\int_{Y_F\otimes \tau V_F}\phi(y')\ \d y'.\end{equation}

The natural identifications $\BY^\square=\BY_\tau^\square$ and $C_c^\infty(\BY^\square)=C_c^\infty(\BY_\tau^\square)$ are used implicitly throughout.

\begin{lem}\label{lem int-val1}
For any $\phi\in C_c^\infty(\BY^\square)$, we have
$$\CI_{\FW}(\phi)(0)=\CI_{\FV_\tau}(\phi)(0).$$
\end{lem}

\begin{proof}
This is immediate from (\ref{equ int-val1}) and (\ref{equ int-val2}).
\end{proof}

\section{Local theta correspondence}\label{sec theta}

\subsection{Theta lifts}
We have introduced the Weil representation $\omega_{V,W,\psi}$ of the dual pair $(\O(V),\Sp(W))$. For $\pi\in\Irr(\Sp(W))$,  its \emph{big theta lift}  $\Theta_{V,W,\psi}(\pi)\in\Rep(\O(V))$ is defined as $$\Theta_{V,W,\psi}(\pi):=(\omega_{V,W,\psi}\otimes \pi^\vee)_{\Sp(W)},$$ where $(\textrm{-})_{\Sp(W)}$ denotes the space of $\Sp(W)$-coinvariants. In other words, $\pi\otimes\Theta_{V,W,\psi}(\pi)$ is the maximal $\pi$-isotypic quotient of $\omega_{V,W,\psi}$. The \emph{small theta lift} $\theta_{V,W,\psi}(\pi)$ is defined to be the maximal semisimple quotient of $\Theta_{V,W,\psi}(\pi)$. 

The same construction applied to $\sigma\in\Irr(\O(V))$ yields theta lifts $\Theta_{V,W,\psi}(\sigma)$ and $\theta_{V,W,\psi}(\sigma)$, which are representations of $\Sp(W)$.

\subsection{The $\tau$-twist}
We may replace $V$ by $V_\tau$ to obtain the big theta lift $\Theta_{V_\tau,W,\psi}(\pi)$, along with the corresponding small theta lift. Likewise, replacing $\psi$ by $\psi_\tau$ yields analogous big and small theta lifts. Recall that we have a natural identification $\O(V_\tau)=\O(V)$. The relations among these various theta lifts are as follows. 

\begin{lem}\label{lem theta relation}
For $\pi\in\Irr(\Sp(W))$, we have 
$$\Theta_{V,W,\psi}(\pi_\tau)=\Theta_{V,W,\psi_\tau}(\pi)=\Theta_{V_\tau,W,\psi}(\pi),\quad \theta_{V,W,\psi}(\pi_\tau)=\theta_{V,W,\psi_\tau}(\pi)=\theta_{V_\tau,W,\psi}(\pi).$$  
\end{lem}

\begin{proof}
This follows directly  from  (\ref{equ weil equ4}).
\end{proof}

\subsection{Seesaw identities} We will use seesaw identities to relate periods on different groups. Let $\Theta_{V_F,\FW,\psi_F}(1)\in\Rep(\Sp(\FW))$ be the big theta  lift of the trivial representation of $\O(V_F)$ with respect to the Weil representation $\omega_{V_F,\FW,\psi_F}$. Similarly, let  $\Theta_{\FV_\tau,W_F,\psi_F}(1)\in\Rep(\O(\FV_\tau))$ be the big theta  lift of the trivial representation of  $\Sp(W_F)$ with respect to the Weil representation $\omega_{\FV_\tau,W_F,\psi_F}$.

\begin{lem}[Base change seesaw identities]\label{lem seesaw} For $\pi\in\Irr(\Sp(W))$ and $\sigma\in\Irr(\O(V))$, we have
$$\Hom_{\O(V_F)}\left(\Theta_{V,W,\psi}(\pi),\BC\right)\cong\Hom_{\Sp(W)}\left(\Theta_{V_F,\FW,\psi_F}(1),\pi\right),$$ and
$$\Hom_{\Sp(W_F)}\left(\Theta_{V_\tau,W,\psi}(\sigma),\BC\right)\cong\Hom_{\O(V)}\left(\Theta_{\FV_\tau,W_F,\psi_F}(1),\sigma\right).$$
\end{lem}

\begin{proof}
This follows from (\ref{equ res-weil 2}), (\ref{equ res-weil 3}), and the definition of the big theta lifts.
\end{proof}

\subsection{Howe duality}
The following fundamental result, known as the Howe duality, was conjectured by Howe \cite{how79, how89}. It was proved by Waldspurger \cite{wal90}  when $p\neq 2$, and in full generality by Gan--Takeda \cite{gt16} and Gan--Sun \cite{gs17}. 

\begin{thm}[Howe duality]\label{thm howe}
For $\pi\in\Irr(\Sp(W))$, the small theta lift $\theta_{V,W,\psi}(\pi)$ is either zero or irreducible. Moreover, if $\theta_{V,W,\psi}(\pi)\cong\theta_{V,W,\psi}(\pi')$
is non-zero, then $\pi\cong\pi'$. The same properties hold for small theta lifts from the orthogonal groups to the symplectic groups.
\end{thm}

Let $\CV$ be the Witt tower of quadratic spaces containing $V$. For $\pi\in\Irr(\Sp(W))$, define its \emph{first occurrence index} in $\CV$ as
$$m_\CV(\pi):=\min\{\dim V'\mid V'\in\CV, \Theta_{V',W,\psi}(\pi) \neq0\}.$$  When $\dim V'=m_\CV(\pi)$, the representation 
$\Theta_{V',W,\psi}(\pi)$ is called the \emph{first occurrence} (of $\pi$ in $\CV$). 
A key result, the {\em persistence principle}, asserts that for any $V'\in\CV$ with $\dim V'\geq m_\CV(\pi)$, one has $ \Theta_{V',W,\psi}(\pi)\neq0$.

Let $V_\rc$ be the \emph{companion} space of $V$; that is, the unique (up to isomorphism) quadratic space over $E$  satisfying $$\dim V_\rc=\dim V,\quad \disc(V_\rc)=\disc(V),\quad  V_\rc\ncong V.$$ Denote by $\CV_\rc$ the Witt tower of quadratic spaces containing $V_\rc$. 

Similarly, let $\CW$ be the Witt tower of symplectic spaces containing $W$. For $\sigma\in\Irr(\O(V))$, define its first occurrence index in $\CW$ as
$$m_\CW(\sigma):=\min\{\dim W'\mid W'\in\CW, \Theta_{V,W',\psi}(\sigma)\neq0\}.$$ 

Recall that $\dim W=2n$ and $\dim V=2m$.

\begin{thm}[Conservation relation, \cite{sz15}]\label{thm conservation}
For $\pi\in\Irr(\Sp(W))$, we have 
$$m_{\CV}(\pi)+m_{\CV_\rc}(\pi)=4n+4.$$
For $\sigma\in\Irr(\O(V))$, we have
$$m_\CW(\sigma)+m_\CW(\sigma\otimes\det)=4m.$$ 
\end{thm}

As a consequence, we have the following \emph{theta 
dichtomy}:
\begin{itemize}
\item For $\pi\in\Irr(\Sp(W))$, if $V_1\in\CV$ and $V_2\in\CV_\rc$ satisfy $$\dim V_1+\dim V_2=4n+2,$$ then exactly one of $\Theta_{V_1,W,\psi}(\pi)$ and $\Theta_{V_2,W,\psi}(\pi)$ is non-zero.
\item For $\sigma\in\Irr(\O(V))$, if $W_1, W_2\in\CW$ satisfy $$\dim W_1+\dim W_2=4m-2,$$ then exactly one of $\Theta_{V,W_1,\psi}(\sigma)$ and $\Theta_{V,W_2,\psi}(\sigma\otimes\det)$ is non-zero.
\end{itemize}

\begin{rem} We collect here some results concerning the equality between big and small theta lifts. Let $\pi\in\Irr(\Sp(W))$ and 
$\sigma\in\Irr(\O(V))$.
\begin{enumerate} 
\item Assume that $\pi$ is supercuspidal. A classical result of Kudla \cite{kud86} states that if $\Theta_{V,W,\psi_\tau}(\pi)$ is the first occurrence, then $\Theta_{V,W,\psi_\tau}(\pi)=\theta_{V,W,\psi_\tau}(\pi)$ and this representation is supercuspidal. The analogous statement also holds for $\sigma$.
 \item Assume that $\pi$ is square-integrable. 
 \begin{itemize}
\item If $m=n$ and $\Theta_{V,W,\psi_\tau}(\pi)\neq0$, then $\Theta_{V,W,\psi_\tau}(\pi)=\theta_{V,W,\psi_\tau}(\pi)$  and it is square-integrable. 
\item If $m=n+1$ and $\Theta_{V,W,\psi_\tau}(\pi)$ is the first occurrence, then $\Theta_{V,W,\psi_\tau}(\pi)=\theta_{V,W,\psi_\tau}(\pi)$ and it is square-integrable. 
\end{itemize} (See \cite[Corollary C.3]{gi14}.)
 \item Assume that $\pi$ is tempered. If $m=n+1$ and $\Theta_{V,W,\psi_\tau}(\pi)\neq0$, then $\Theta_{V,W,\psi_\tau}(\pi)=\theta_{V,W,\psi_\tau}(\pi)$ and it is tempered. (See \cite[Proposition C.4]{gi14}.)
 \item Assume that $\sigma$ is square-integrable. 
 \begin{itemize}
\item If $m=n+1$ and $\Theta_{V,W,\psi_\tau}(\sigma)\neq0$, then $\Theta_{V,W,\psi_\tau}(\sigma)=\theta_{V,W,\psi_\tau}(\sigma)$ and it is square-integrable. 
\item If $m=n$ and $\Theta_{V,W,\psi_\tau}(\sigma)$ is the first occurrence, then $\Theta_{V,W,\psi_\tau}(\sigma)=\theta_{V,W,\psi_\tau}(\sigma)$ and it is square-integrable.\end{itemize} (See \cite[Corollary C.3]{gi14}.) 
 \item Assume that $\sigma$ is tempered. If $m=n$ and $\Theta_{V,W,\psi_\tau}(\sigma)\neq0$, then $\Theta_{V,W,\psi_\tau}(\sigma)=\theta_{V,W,\psi_\tau}(\sigma)$ and it is tempered. (See \cite[Proposition C.4]{gi14}.)
\end{enumerate}
\end{rem}

\section{Degenerate principal series}\label{sec deg}

\subsection{Degenerate principal series and Rallis map}
Let $\FP$ be the Siegel parabolic subgroup of $\FG:=\Sp(\FW)$ stabilizing  the Lagrangian $\FX'=X_F\oplus \tau Y_F$. Then $\FP$ has a Levi decomposition $\FP=\FM\FN$ with $\FM\cong\GL(\FX')$. Let  $\det:\GL(\FX')\ra F^\times$ denotes the determinant map. For $s\in\BC$, the \emph{degenerate principal series} is defined as the normalized parabolic induction  
$$I_\FP^\FG(s):=\Ind_\FP^\FG(|\det|_F^s).$$ We set  
\begin{equation}\label{equ s0}s_0:=\frac{2m-2n-1}{2}.\end{equation}

Consider the Weil representation $\omega_{V_F,\FW,\psi_F}$ of $\FG\times\O(V_F)$, which is realized on the Schrodinger model $C_c^\infty(\BY^{\square'})$. The action of $\O(V_F)$ on this model is given by $$(\omega_{V_F,\FW,\psi_F}(h)\phi)(x)=\phi(h^{-1}x),\quad h\in\O(V_F),\ \phi\in C_c^\infty(\BY^{\square'}).$$ 
The \emph{Rallis map} 
$$R_\FW: C_c^\infty(\BY^{\square'})\lra I_\FP^\FG(s_0),\quad \phi\mapsto\CF_\phi,$$ is defined by
\begin{equation}\label{equ rallis-defn}\CF_\phi(g):=(\omega_{V_F,\FW,\psi_F}(g)\phi)(0),\quad g\in\FG.\end{equation} Let $R_\FW(V_F)$ denote the image of $R_\FW$.
It is known that $R_\FW$ is  $\FG$-equivariant, $\O(V_F)$-invariant, and
\begin{equation}\label{equ rallis1}R_\FW(V_F)\cong\Theta_{V_F,\FW,\psi_F}(1).\end{equation}

Similarly, let $\FQ$ be the Siegel parabolic subgroup of $\FH:=\O(\FV_\tau)$ stabilizing the Lagrangian $\tau V_F$, and $I_\FQ^\FH(|\det|_F^s)$ the degenerate principal series of $\FH$. In this setting, the Rallis map 
$$R_{\FV_\tau}: C_c^\infty(\BY_\tau^{\square'})\lra I_\FQ^\FH(-s_0),\quad\phi\mapsto\CF_\phi,$$ is defined by
$$\CF_\phi(h):=(\omega_{\FV_\tau,W_F,\psi_F}(h)\phi)(0),\quad h\in\FH.$$
Let $R_{\FV_\tau}(W_F)$ denote the image of $R_{\FV_\tau}$. It satisfies \begin{equation}\label{equ rallis2}R_{\FV_\tau}(W_F)\cong\Theta_{\FV_\tau,W_F,\psi_F}(1).\end{equation}

The following is a direct consequence of the seesaw identities (Lemma \ref{lem seesaw}).

\begin{cor}\label{cor seesaw}
For $\pi\in\Irr(\Sp(W))$ and $\sigma\in\Irr(\O(V))$, we have
$$\Hom_{\O(V_F)}\left(\Theta_{V,W,\psi}(\pi),\BC\right)\cong\Hom_{\Sp(W)}\left(R_{\FW}(V_F),\pi\right),$$ and
$$\Hom_{\Sp(W_F)}\left(\Theta_{V_\tau,W,\psi}(\sigma),\BC\right)\cong\Hom_{\O(V)}\left(R_{\FV_\tau}(W_F),\sigma\right).$$
\end{cor}

\subsection{Structures of the degenerate principal series I}
We now recall the structure of the degenerate principal series as a $\FG$-module or $\FH$-module; a summary can be found in \cite[Sect. 7]{gi14}.

Let $V_{\rc,F}$ be the companion space of $V_F$. Associated to $V_{\rc,F}$, we have the big theta lift $\Theta_{V_{\rc,F},\FW,\psi_F}(1)$ and the image $R_\FW(V_{\rc,F})$ of the Rallis map. Moreover,  $$\Theta_{V_{\rc,F},\FW,\psi_F}(1)\cong R_\FW(V_{\rc,F})\subset I_\FP^\FG(s_0).$$ 

Let $\wt{V}_F$ be the \emph{complementary} space of $V_F$; that is, the unique quadratic space in the Witt tower of $V_F$ satisfying  $$\dim \wt{V}_F+\dim V_F=4n+2.$$ 
Similarly, let $\wt{V}_{\rc,F}$ be the complementary  space of $V_{\rc,F}$. This yields $\FG$-submodules $R_\FW(\wt{V}_F)$ and $R_\FW(\wt{V}_{\rc,F})$ of $I_\FP^\FG(-s_0)$.
 
\begin{prop}\label{prop degps-symp}
\begin{enumerate}
	\item If  $m>n$, then 
	\begin{enumerate}
	\item $I_\FP^\FG(s_0)=R_\FW(V_F)+R_\FW(V_{\rc,F})$; 
	\item $R_\FW(V_F)\cap R_\FW(V_{\rc,F})$ is irreducible and is the maximal semisimple submodule of $I_\FP^\FG(s_0)$;
	\item  we have
	$$R_\FW(V_F)/R_\FW(V_F)\cap R_\FW(V_{\rc,F})\cong R_\FW(\wt{V}_F)$$ and $$R_\FW(V_{\rc,F})/R_\FW(V_F)\cap R_\FW(V_{\rc,F})\cong R_\FW(\wt{V}_{\rc,F});$$
	\item $R_\FW(\wt{V}_F)$ and $R_\FW(\wt{V}_{\rc,F})$ are irreducible.  
	\end{enumerate}  In particular, we have the exact sequences
	$$
	0\lra R_\FW(V_F)\lra I_\FP^\FG(s_0)\lra R_\FW(\wt{V}_{\rc,F})\lra0$$ and $$
	0\lra R_\FW(V_{\rc,F})\lra I_\FP^\FG(s_0)\lra R_\FW(\wt{V}_F)\lra0.
	$$
	\item If $m\leq n$,  then 
	\begin{enumerate}
	\item $R_\FW(V_F)$, $R_\FW(V_{\rc,F})$ and $R_\FW(\wt{V}_F)\cap R_\FW(\wt{V}_{\rc,F})$ are all irreducible; 
	\item $R_\FW(V_F)\oplus R_\FW(V_{\rc,F})$ is the maximal semisimple submodule of $I_\FP^\FG(s_0)$;
	\item we have the exact sequence
	$$0\lra R_\FW(V_F)\oplus R_\FW(V_{\rc,F})\lra I_\FP^\FG(s_0)\lra R_\FW(\wt{V}_F)\cap R_\FW(\wt{V}_{\rc,F})\lra0.$$
	\end{enumerate}
	\end{enumerate}
\end{prop}

Let $\wt{W}_F$ be the unique symplectic space in the Witt tower of $W_F$ satisfying $$\dim \wt{W}_F+\dim W_F=4m-2.$$
This gives rise to an $\FH$-submodule $R_{\FV_\tau}(\wt{W}_F)$ of $I_\FQ^\FH(s_0)$.

\begin{prop}\label{prop degps-orth}
\begin{enumerate}
\item If $m\leq n$, then
\begin{enumerate}
\item $I_\FQ^\FH(-s_0)=R_{\FV_\tau}(W_F)+R_{\FV_\tau}(W_F)\otimes\det$; 
\item $R_{\FV_\tau}(W_F)\cap R_{\FV_\tau}(W_F)\otimes\det$ is irreducible and is the maximal semisimple submodule of $I_\FQ^\FH(-s_0)$; 
\item we have $$R_{\FV_\tau}(W_F)/R_{\FV_\tau}(W_F)\cap R_{\FV_\tau}(W_F)\otimes\det\cong R_{\FV_\tau}(\wt{W}_F)$$ and  $$R_{\FV_\tau}(W_F)\otimes\det/R_{\FV_\tau}(W_F)\cap R_{\FV_\tau}(W_F)\otimes\det\cong R_{\FV_\tau}(\wt{W}_F)\otimes\det.$$ 
\item $R_{\FV_\tau}(\wt{W}_F)$ is irreducible.  
\end{enumerate}
In particular, we have the exact sequences
$$0\lra R_{\FV_\tau}(W_F)\lra I_\FQ^\FH(-s_0)\lra R_{\FV_\tau}(\wt{W}_F)\otimes\det\lra0$$ and
$$0\lra R_{\FV_\tau}(W_F)\otimes\det\lra I_\FQ^\FH(-s_0)\lra R_{\FV_\tau}(\wt{W}_F)\lra0.$$
\item If $m>n$, then
\begin{enumerate}
\item$R_{\FV_\tau}(W_F)$ and $R_{\FV_\tau}(\wt{W}_F)\cap R_{\FV_\tau}(\wt{W}_F)\otimes\det$ are irreducible; 
\item $R_{\FV_\tau}(W_F)\bigoplus R_{\FV_\tau}(W_F)\otimes\det$ is the maximal semisimple submodule of $I_\FQ^\FH(-s_0)$;
\item we have the exact sequence
$$0\lra R_{\FV_\tau}(W_F)\oplus R_{\FV_\tau}(W_F)\otimes\det\lra I_\FQ^\FH(-s_0)\lra R_{\FV_\tau}(\wt{W}_F)\cap R_{\FV_\tau}(\wt{W}_F)\otimes\det\lra0.$$
\end{enumerate}
\end{enumerate}
\end{prop}

\subsection{Structures of the degenerate principal series II}\label{subsec deg-principal-II} Set $G=\Sp(W)$ and $H=\O(V)$.
In this subsection, we study the structure of the degenerate principal series $I_\FP^\FG(s)$ (resp. $I_\FQ^\FH(s)$) as a $G$-module (resp. $H$-module).

Let $r$ be the Witt index of $V$. Recall from Sect. \ref{subsec explicit orbit} that $\CX=\FQ\bs\FH$, viewed as the set of Lagrangians of $\FV_\tau$, admits as a disjoint decomposition
$$\CX=\bigsqcup_{k=0}^r\CX_r,\quad \CX_k=\left\{L\in\CX\mid\dim_EL_\an=k\right\}.$$ A direct verification shows that
$$\overline{\CX}_k=\bigsqcup_{k'\geq k}\CX_{k'},$$ where $\overline{\CX}_k$ denotes the closure of $\CX_k$ in $\CX$. 

Applying the Geometric Lemma of Bernstein--Zelevinsky, this stratification yields an $H$-filtration 
$$0=I_0(s)\subset I_1(s)\subset\cdots\subset I_{r+1}(s)=I_\FQ^\FH(s),$$ with
$$I_k(s):=\left\{f\in I_\FQ^{\FH}(s)\mid f|_{\overline{\CX}_k}=0  \right\}.$$ For $ 0\leq k\leq r$, the successive quotients are denoted by
$$J_k(s):=I_{k+1}(s)/I_k(s).$$ 

Recall from Lemma \ref{lem orbits-coho2} (see Sect. \ref{subsec explicit orbit} for notation) that
$$\CX_k=\bigsqcup_{V'_{2m-2k,F}\in\RH^1(E/F,\O(V_{2m-2k}))} H\cdot\left(X_k\oplus \tau V'_{2m-2k,F}\right),$$ and from Lemma \ref{lem stab_O} that the stabilizer of $L_k:=X_k\oplus \tau V'_{2m-2k,F}$ is $$R_k:=\Stab_{H}(L_k)=\left(\GL(X_k)\times\O(V'_{2m-2k,F})\right)\ltimes N,$$ which is contained in the parabolic subgroup
$$Q_k:=\Stab_H(X_k)=\left(\GL(X_k)\times\O(V_{2m-2k})\right)\ltimes N.$$ Now suppose that $L_k=g_k^{-1}\cdot\tau V_F$ for some $g_k\in\FQ\bs\FH$. Then $g_kR_kg_k^{-1}\subset\FQ$. For a representation $\chi$ of $\FQ$, we define a representation $\chi^{g_k}$  of $R_k$ by $$\chi^{g_k}(r)=\chi(g_krg_k^{-1}),\quad r\in R_k.$$

\begin{prop}\label{prop degps-orth-2}
For $H=\O(V)$ and $0\leq k\leq r$, we have
$$J_k(s)\cong\bigoplus_{V'_{2m-2k,F}\in\RH^1(E/F,\O(V_{2m-2k}))} \Ind_{Q_k}^H\left(|\det{_{X_k}}|_E^{s+\frac{k}{2}}\otimes C_c^\infty\left(\O(V_{2m-2k,F})\bs\O(V_{2m-2k})\right)\right).$$
In particular, 
$$J_0(s)=I_1(s)\cong\bigoplus_{V'_F\in\RH^1(E/F,\O(V))}C_c^\infty\left(\O(V'_F)\bs\O(V)\right)$$ is independent of $s$.
\end{prop}

\begin{proof}
Let $\chi_s$ denote the character $|\det|_F^s:\FQ\ra\GL(\tau V_F)\ra\BC^\times$. Then we have
$$\begin{aligned}J_k(s)&\cong\bigoplus_{V'_{2m-2k,F}\in\RH^1(E/F,\O(V_{2m-2k}))} \textrm{un-Ind}_{R_{k}}^H\left((\chi_s\delta_\FQ^{1/2})^{g_k}\right)\\
&\cong\bigoplus_{V'_{2m-2k,F}\in\RH^1(E/F,\O(V_{2m-2k}))} \textrm{un-Ind}_{Q_{k}}^H\left(\textrm{un-Ind}_{R_k}^{Q_k}\left((\chi_s\delta_\FQ^{1/2})^{g_k}\right)\right)\\
&\cong\bigoplus_{V'_{2m-2k,F}\in\RH^1(E/F,\O(V_{2m-2k}))} \Ind_{Q_{k}}^H\left((\chi_s\delta_\FQ^{1/2}\delta_{Q_k}^{-1/2})|_{\GL(X_k)}\otimes C_c^\infty\left(\O(V'_{2m-2k,F})\bs\O(V_{2m-2k})\right)\right),
\end{aligned}$$ where $\textrm{un-Ind}$  denotes the unnormalized induction.

For $a\in\GL(X_k)$, we compute
$$\begin{aligned}(\chi_s\delta_\FQ^{1/2}\delta_{Q_k}^{-1/2})(a)&=
|\N_{E/F}(\det(a))|_F^s\cdot|\N_{E/F}(\det(a))|_F^{\frac{4m-2m-1}{2}}\cdot|\det(a)|_E^{-\frac{2m-k-1}{2}}\\
&=|\det(a)|_E^{s+\frac{k}{2}}.
\end{aligned}$$
Thus the proposition follows.\end{proof}

In the same way, $I_\FP^\FG(s)$ possesses a $G$-filtration 
$$0=I_0(s)\subset I_1(s)\subset\cdots\subset I_{n+1}(s)=I_\FP^\FG(s)$$ with successive quotients
$J_k(s)=I_{k+1}(s)/I_k(s)$ (for $0\leq k\leq n$) described as follows.

\begin{prop}\label{prop degps-symp-2}
For $G=\Sp(W)$ and $0\leq k\leq n$, we have
$$J_k(s)=\Ind_{P_k}^G\left(|\det{_{X_k}}|_E^{s+\frac{k}{2}}\otimes C_c^\infty\left(\Sp(W_{F,2n-2k})^\tau\bs\Sp(W_{2n-2k})\right)\right),$$ where $X_k$ is the isotropic subspace of $W$ defined in Sect. \ref{subsec explicit orbit} and $P_k$ is the parabolic subgroup of $G$ that stablizes $X_k$. In particular, $$J_0(s)=I_1(s)\cong C_c^\infty\left(\Sp(W_F)^\tau\bs\Sp(W) \right)$$ is independent of $s$.
\end{prop}

\section{Multiplicities}\label{sec mult}
For $\pi\in\Rep(\Sp(W))$ and $\sigma\in\Rep(\O(V))$), we set, for brevity, 
$$m_{W_F}(\pi):=\dim\Hom_{\Sp(W_F)}(\pi,\BC),\quad  m_{V'_F}(\sigma):=\dim\Hom_{\O(V'_F)}(\sigma,\BC),$$ where $V'_F\in\RH^1(E/F,\O(V))$.

\subsection{The case $m>n$}   
\begin{thm}\label{thm symp-orth} Let $\pi\in\Irr(\Sp(W))$.
Assume that $m>n$ and that $\Theta_{V,W,\psi_\tau}(\pi)$ is the first occurrence. 
\begin{enumerate}
\item If $\pi$ is tempered, then 
$$m_{W_F}(\pi^\vee)\leq m_{V_F}\left(\Theta_{V,W,\psi_\tau}(\pi)\right).$$
\item If $\pi$ is square-integrable and $m=n+1$, then
$$m_{W_F}(\pi^\vee)= m_{V_F}\left(\Theta_{V,W,\psi_\tau}(\pi)\right).$$
\item If $\pi$ is supercuspidal, then $$m_{W_F}(\pi^\vee)= m_{V_F}\left(\Theta_{V,W,\psi_\tau}(\pi)\right).$$
\end{enumerate}
\end{thm}

\begin{rem}\begin{enumerate}
\item Recall from Lemma \ref{lem theta relation}  that $\Theta_{V,W,\psi_\tau}(\pi)=\Theta_{V,W,\psi}(\pi_\tau)$. 
A direct verification shows $$\Hom_{\Sp(W_F)}(\pi,\BC)=\Hom_{\Sp(W_F)^\tau}\left(\pi_\tau,\BC\right),\quad (\pi^\vee)_\tau=(\pi_\tau)^\vee.$$ Consequently, we may unambiguously write  $\pi_\tau^\vee$ for either expression, and we obtain
$$\Hom_{\Sp(W_F)}\left(\pi^\vee,\BC\right)=\Hom_{\Sp(W_F)^\tau}\left(\pi^\vee_\tau,\BC\right).$$
\item Note that if $\pi$ is supercuspidal (resp. square-integrable, tempered), then the same holds for $\pi_\tau$.
\end{enumerate}
\end{rem}

\begin{proof} 
Under the assumption $m>n$, we have $$\dim V_F=\dim V_{\rc,F}>\dim \wt{V}_F=\dim \wt{V}_{\rc,F}.$$
Set
$$\wt{V}:=\wt{V}_{F}\otimes_FE,\quad {V}_\rc:={V}_{\rc,F}\otimes_FE,\quad \wt{V}_\rc:=\wt{V}_{\rc,F}\otimes_FE.$$
According to Proposition \ref{prop degps-symp} and Corollary \ref{cor seesaw}, we obtain the exact sequence
\begin{equation}\label{equ exact-1}\begin{aligned}
0\lra\Hom_{\O(\wt{V}_{\rc,F})}\left(\Theta_{\wt{V}_{\rc},W,\psi}(\pi_\tau),\BC\right)\lra\Hom_{\Sp(W)}\left(I_\FP^\FG(s_0),\pi_\tau\right)\lra\Hom_{\O(V_F)}\left(\Theta_{V,W,\psi}(\pi_\tau),\BC\right),
\end{aligned}\end{equation} where the last arrow is surjective if $\pi$ is supercuspidal.

Since $V$ and $\wt{V}$ lie in the same Witt tower and $\Theta_{V,W,\psi}(\pi_\tau)$ is the first occurrence, we have $$\Theta_{\wt{V},W,\psi}(\pi_\tau)=0.$$ If $V$ and $V_\rc$ are isomorphic, then
$\Theta_{V_\rc,W,\psi}(\pi_\tau)\neq0$, and consequently $$\Theta_{\wt{V}_\rc,W,\psi}(\pi_\tau)=0.$$ If $V$ and $V_\rc$ are not isomorphic, since $\dim V+\dim\wt{V}_\rc=4n+2$, by the theta dichotomy, we see 
$$\Theta_{\wt{V}_\rc,W,\psi}(\pi_\tau)=0.$$ Thus, in all cases, the exact sequence (\ref{equ exact-1}) yields an injection \begin{equation}\label{equ inj1}\Hom_{\Sp(W)}\left(I_\FP^\FG(s_0),\pi_\tau\right)\lhra\Hom_{\O(V_F)}(\Theta_{V,W,\psi}(\pi_\tau),\BC),\end{equation} 
which is an isomorphism when $\pi$ is supercuspidal.

Replacing $s_0$ by $-s_0$ in Proposition \ref{prop degps-symp} yields the exact sequence
$$0\lra R_\FW(\wt{V}_F)\oplus R_\FW(\wt{V}_{\rc,F})\lra I_\FP^\FG(-s_0)\lra R_\FW(V_F)\cap R_\FW(V_{\rc,F})\lra0.$$ 
Since $\Theta_{\wt{V},W,\psi}(\pi_\tau)$ and $\Theta_{\wt{V}_\rc,W,\psi}(\pi_\tau)$ are both vanish, Corollary \ref{cor seesaw} gives
\begin{equation}\label{equ isom1}\Hom_{\Sp(W)}\left(R_\FW(V_F)\cap R_\FW(V_{\rc,F}),\pi_\tau\right)\cong\Hom_{\Sp(W)}\left(I_\FP^\FG(-s_0),\pi_\tau\right).\end{equation}
Combining  the isomorphism
	$$R_\FW(V_F)/R_\FW(V_F)\cap R_\FW(V_{\rc,F})\cong R_\FW(\wt{V}_F) 
	$$ with Corollary \ref{cor seesaw}, we obtain an injection
	\begin{equation}\label{equ inj2}\Hom_{\O(V_F)}(\Theta_{V,W,\psi}(\pi_\tau),\BC)\lhra \Hom_{\Sp(W)}\left(R_\FW(V_F)\cap R_\FW(V_{\rc,F}),\pi_\tau\right),\end{equation} 
which is an isomorphism when $\pi$ is supercuspidal.

Combining the injection (\ref{equ inj1}), the isomorphism (\ref{equ isom1}) and  the injection (\ref{equ inj2}), we obtain a chain of injections
\begin{equation}\label{equ inj3}\Hom_{\Sp(W)}\left(I_\FP^\FG(s_0),\pi_\tau\right)\lhra\Hom_{\O(V_F)}(\Theta_{V,W,\psi}(\pi_\tau),\BC)\lhra\Hom_{\Sp(W)}\left(I_\FP^\FG(-s_0),\pi_\tau\right),\end{equation} 
 all of which are isomorphisms when $\pi$ is supercuspidal.
 
We now discuss the relation between $$\Hom_{\Sp(W)}\left(I_\FP^\FG(s),\pi_\tau\right)\quad \textrm{and}\quad \Hom_{\Sp(W_F)}\left(\pi^\vee,\BC\right).$$ Recall from Sect. \ref{subsec deg-principal-II}  that $I_\FP^\FG(s)$ admits an $\Sp(W)$-equivariant filtration with successive quotients $J_k(s)$ for $0\leq k\leq n$. 
 
 For $k=0$, we have $$\begin{aligned}\Hom_{\Sp(W_F)}\left(\pi^\vee,\BC\right)&=\Hom_{\Sp(W_F)^\tau}\left(\pi_\tau^\vee,\BC\right)\\&\cong\Hom_{\Sp(W)}\left(C_c^\infty\left(\Sp(W_F)^\tau\bs\Sp(W)\right),\pi_\tau\right)\\
 &\cong \Hom_{\Sp(W)}\left(J_0(s),\pi_\tau\right).
 \end{aligned}.$$  
 
For $0<k\leq n$, Proposition \ref{prop degps-symp-2} together with Bernstein's Frobenius reciprocity yields
 $$ \Hom_{\Sp(W)}\left(J_k(s),\pi_\tau\right)\cong\Hom_{\GL(X_k)\times\Sp(W_{2n-2k})}\left(|\det\nolimits_{X_k}|_E^{s+\frac{k}{2}}\otimes C_c^\infty\left(\Sp(W_{2n-2k,F})^\tau\bs \Sp(W_{2n-2k})\right),R_{\overline{P}_k}(\pi_\tau)\right)$$ and
  $$ \Ext^1_{\Sp(W)}\left(J_k(s),\pi_\tau\right)\cong\Ext^1_{\GL(X_k)\times\Sp(W_{2n-2k})}\left(|\det\nolimits_{X_k}|_E^{s+\frac{k}{2}}\otimes C_c^\infty\left(\Sp(W_{2n-2k,F})^\tau\bs \Sp(W_{2n-2k})\right),R_{\overline{P}_k}(\pi_\tau)\right).$$

If $\pi$ is square-integrable (resp. tempered),  Casselman's criterion implies that the center of $\GL(X_k)$ acts on any irreducible subquotient of $R_{\overline{P}_k}(\pi)$ by a character of the form $\mu|\cdot|^\alpha$ with $\mu$ unitary and $\alpha<0$ (resp. $\alpha\leq 0$). Consequently, 
\begin{itemize}
\item for tempered $\pi$ and every $0<k\leq n$, we have
$$\Hom_{\Sp(W)}\left(J_k(s_0),\pi_\tau\right)=0,\quad\Ext^1_{\Sp(W)}\left(J_k(s_0),\pi_\tau\right)=0;$$ 
\item if $\pi$ is square-integrable and $m=n+1$, then for all $0<k\leq n$, we have
$$\Hom_{\Sp(W)}\left(J_k(-s_0),\pi_\tau\right)=0,\quad \Ext^1_{\Sp(W)}\left(J_k(-s_0),\pi_\tau\right)=0.$$
\end{itemize}
Therefore, for tempered $\pi$ we have 
\begin{equation}\label{equ pd-isom1}\Hom_{\Sp(W_F)}\left(\pi^\vee,\BC\right)\cong\Hom_{\Sp(W)}\left(I_\FP^\FG(s_0),\pi_\tau\right),\end{equation} and for square-integrable $\pi$ with $m=n+1$,
\begin{equation}\label{equ pd-isom2}\Hom_{\Sp(W_F)}\left(\pi^\vee,\BC\right)\cong\Hom_{\Sp(W)}\left(I_\FP^\FG(-s_0),\pi_\tau\right).\end{equation} Finally, for supercuspidal $\pi$ one trivially has
\begin{equation}\label{equ pd-isom3}\Hom_{\Sp(W_F)}\left(\pi^\vee,\BC\right)\cong\Hom_{\Sp(W)}\left(I_\FP^\FG(s),\pi_\tau\right), \forall s\in\BC.\end{equation} 

Combining \eqref{equ inj3} with the isomorphisms \eqref{equ pd-isom1}--\eqref{equ pd-isom3} yields Theorem \ref{thm symp-orth}.

\end{proof}

\begin{rem}
We expect that if $\pi$ is $\Sp(W_F)$-distinguished, then $\Hom_{\Sp(W)}\left(I_\FP^\FG(s),\pi_\tau\right)\neq0$ for every $s\in\BC$. If this holds, then (\ref{equ inj1}) directly implies that 
$\Theta_{V,W,\psi_\tau}(\pi)$ is $\O(V_F)$-distinguished. 
 \end{rem}

\subsection{The case $m\leq n$}

\begin{thm}\label{thm orth-symp}
Let $\sigma\in\Irr(\O(V))$. Assume that $m\leq n$ and that $\Theta_{V,W,\psi_\tau}(\sigma)$ is the first occurrence. 
\begin{enumerate}
\item If $\sigma$ is tempered, then 
$$m_{W_F}\left(\Theta_{V,W,\psi_\tau}(\sigma)\right)\geq\sum_{V'_F\in \RH^1(E/F,\O(V))}m_{V'_F}\left(\sigma^\vee\right).$$
\item If $\sigma$ is square-integrable and $m=n$, then
$$m_{W_F}\left(\Theta_{V,W,\psi_\tau}(\sigma)\right)=\sum_{V'_F\in \RH^1(E/F,\O(V))}m_{V'_F}\left(\sigma^\vee\right).$$
\item If $\sigma$ is supercuspidal, then $$m_{W_F}\left(\Theta_{V,W,\psi_\tau}(\sigma)\right)=\sum_{V'_F\in \RH^1(E/F,\O(V))}m_{V'_F}\left(\sigma^\vee\right).$$
\end{enumerate}
\end{thm}

\begin{proof}
The proof is similar to that of Theorem \ref{thm symp-orth}. We only provide a sketch. 

According to Proposition \ref{prop degps-orth}, we have the exact sequence
$$0\lra\Hom_{\O(V)}\left(R_{\FV_\tau}(\wt{W}_F)\otimes\det,\sigma \right)\lra\Hom_{\O(V)}\left(I_\FQ^\FH(-s_0),\sigma\right)\lra\Hom_{\O(V)}\left(R_{\FV_\tau}(W_F),\sigma\right).$$
Observe that the restriction of the determinant character $\det$ of $\O(\FV_\tau)$ to the subgroup $\O(V)$ coincides with $\N_{E/F}\circ\det_{\O(V)}$, which is identically 1 on $\O(V)$. Hence
$$\Hom_{\O(V)}\left(R_{\FV_\tau}(\wt{W}_F)\otimes\det,\sigma \right)=\Hom_{\O(V)}\left(R_{\FV_\tau}(\wt{W}_F),\sigma \right).$$
By Corollary \ref{cor seesaw}, the latter space is isomorphic to
$$\Hom_{\Sp(\wt{W}_F)}\left(\Theta_{V,\wt{W},\psi_\tau}(\sigma),\BC\right).$$ Since $\Theta_{V,W,\psi_\tau}(\sigma)$ is assumed to be the first occurrence, this $\Hom$ space is zero. Therefore, from the exact sequence above we obtain an injection $$\Hom_{\O(V)}\left(I_\FQ^\FH(-s_0),\sigma\right)\lhra\Hom_{\Sp(W_F)}\left(\Theta_{V,W,\psi_\tau}(\sigma),\BC\right).$$ 
Replacing $-s_0$ to $s_0$ in Proposition \ref{prop degps-orth} yields another injection
$$\Hom_{\Sp(W_F)}\left(\Theta_{V,W,\psi_\tau}(\sigma),\BC\right)\lhra\Hom_{\O(V)}\left(I_\FQ^\FH(s_0),\sigma\right).$$ 
In summary, we obtain a sequence of injections
\begin{equation}\label{equ inj4}\Hom_{\O(V)}\left(I_\FQ^\FH(-s_0),\sigma\right)\lhra\Hom_{\Sp(W_F)}\left(\Theta_{V,W,\psi_\tau}(\sigma),\BC\right)\lhra\Hom_{\O(V)}\left(I_\FQ^\FH(s_0),\sigma\right).\end{equation}
When $\sigma$ is supercuspidal, both injections in (\ref{equ inj4}) are isomorphisms.

The degenerate principal series $I_\FQ^\FH(s)$ admits 
an $\O(V)$-equivariant filtration  with successive quotients $J_k(s)$ for $0\leq k\leq r$. Applying Proposition \ref{prop degps-orth-2} together with Casselman's criterion yields:
\begin{itemize} 
\item if $\sigma$ is tempered, then 
\begin{equation}\label{equ pd-isom4}\Hom_{\O(V)}\left(J_0(-s_0),\sigma\right)\cong\Hom_{\O(V)}\left(I_\FQ^\FH(-s_0),\sigma\right);\end{equation}  
\item if $\sigma$ is square-integrable and $m=n$, then
\begin{equation}\label{equ pd-isom5}\Hom_{\O(V)}\left(J_0(s_0),\sigma\right)\cong\Hom_{\O(V)}\left(I_\FQ^\FH(s_0),\sigma\right).\end{equation} 
\end{itemize}
For supercuspidal $\sigma$, one trivially has
\begin{equation}\label{equ pd-isom6}\Hom_{\O(V)}\left(J_0(s),\sigma\right)\cong\Hom_{\O(V)}\left(I_\FQ^\FH(s),\sigma\right), \forall s\in\BC.\end{equation} 
On the other hand,  Proposition \ref{prop degps-orth-2} gives $$\Hom_{\O(V)}\left(J_0(s),\sigma\right)\cong\bigoplus_{V'_F\in\RH^1(E/F,\O(V))}\Hom_{\O(V'_F)}(\sigma^\vee,\BC).$$  

Theorem \ref{thm symp-orth} then follows from (\ref{equ inj4}), (\ref{equ pd-isom4}), (\ref{equ pd-isom5}) and (\ref{equ pd-isom6}).

\end{proof}

\section{Transfer of local periods}\label{sec transfer}

\subsection{Harmonic analysis on symmetric spaces} We begin by recalling some basic facts about harmonic analysis on $p$-adic symmetric spaces.

Let $G$ be a reductive group over $F$ and $\iota$ an $F$-involution of $G$. We denote by $G:=G(F)$ the group of $F$-rational points, and by $H:=G^\iota$ the subgroup of points fixed by $\iota$. The homogeneous space $X:=H\bs G$ is a symmetric space. Let $Z$ be the center of $G$.

 For $\alpha\in\Hom_H(\pi,\BC)$ and $v\in\pi$,  the {\em generalized matrix coefficient}  $\varphi_{\alpha,v}: G\ra\BC$ is defined by $$\varphi_{\alpha,v}(g)=\alpha(\pi(g)v).$$ This function is $H$-invariant and therefore descends to a function on the symmetric space $X=H\bs G$.

Let $\pi$ be an $H$-distinguished representation of $G$, and $\alpha\in\Hom_H(\pi,\BC)$ a non-zero local period.  We now recall several analytic properties of $\alpha$ and $\pi$.
\begin{enumerate}
\item The period $\alpha$ is called \emph{relatively supercuspidal} if for every $v\in\pi$ the generalized matrix coefficient $\varphi_{\alpha,v}$ is compactly supported modulo $ZH$. The representation $\pi$ is said to be $H$-\emph{relatively supercuspidal} if every $\alpha\in\Hom_H(\pi,\BC)$ is relatively supercuspidal.
\item Assume that $\pi$ admits a unitary central character. 
\begin{itemize}
\item The period $\alpha$ is \emph{relatively square-integrable} if $\varphi_{\alpha,v}\in L^2(ZH\bs G)$  for all $v\in\pi$. 
\item The period $\alpha$ is  \emph{relatively tempered} if 
$\varphi_{\alpha,v}\in L^{2+\epsilon}(ZH\bs G)$  for  every $\epsilon>0$ and all $v\in\pi$. 
\end{itemize}
The representation $\pi$ is called $H$-\emph{relatively square-integrable} (resp. $H$-\emph{relatively tempered}) if every  
$\alpha\in\Hom_H(\pi,\BC)$ is relatively square-integrable (resp. relatively tempered).
\end{enumerate}

A fundamental fact (see \cite{kt1}, \cite{kt2} and \cite{dh}) is that  any $H$-distinguished representation $\pi$ which is supercuspidal (resp. square-integrable, tempered)  is automatically $H$-relatively supercuspidal (resp. $H$-relatively square-integrable, $H$-relatively tempered).

We recall the following terminology. A split torus $S$ of $G$ is called \emph{$\iota$-split} if $\iota(s)=s^{-1}$ for any $s\in S$. A parabolic subgroup $P$ of $G$ is said to be \emph{$\iota$-parabolic} if $P$ and $\iota(P)$ are opposite. Given an $\iota$-parabolic subgroup $P$, denote by $S_P$ the maximal $\iota$-split torus in the center of $P\cap\iota(P)$, and define
$$S_P^+:=\{s\in S_P\mid |\alpha(s)|_F\leq 1,\ \forall\alpha\in\Delta(S_P,P) \},$$ where $\Delta(S_P,P)$ is the set of simple roots of $S_P$ in the Lie algebra of $P$.
The {\em weak Cartan decomposition} for the symmetric space $X=H\bs G$ (see \cite[Theorem 1.1]{bo}) asserts the existence of a compact subset $\Omega\subset G$ and a finite set $\CP$ of minimal $\iota$-parabolic subgroups of $G$ such that
\begin{equation}\label{equ cartan}G=\bigcup_{P\in\CP}HS_P^+\Omega.\end{equation}
	
Let $\Xi$ be the Harish-Chandra function on $G$, whose basic properties are recalled in \cite[Sect. II.1]{wa1}. Define a function $\Xi^X$ on the symmetric space $X$ by $$\Xi^X(Hg):=\Xi(s(g))^{1/2},\quad g\in G,$$ where $s(g)=\iota(g)^{-1}g$. Fix a height function $\norm{\textrm{-}}$ on $G$. For $d\in\BZ$, set $$\N_d(Hg):=(1+\log\norm{s(g)})^d.$$ For any minimal $\iota$-parabolic subgroup $P\in\CP$, there exist non-negative integers $d,d'\in\BN$ such that (cf. \cite[Proposition 6]{lag})
\begin{equation}\label{equ estimate1}\delta_{P}^{1/2}(a)\N_{-d}(Ha)\ll\Xi^X(Ha)\ll\delta_{P}^{1/2}(a)\N_{d'}(Ha),\quad\forall a\in S_P,\end{equation} 
and (cf. \cite[Proposition 2.6]{kt2})
\begin{equation}\label{equ estimate2}\vol(H\bs Ha\Omega)\asymp\delta_{P}^{-1}(a),\quad\forall a\in S_P^+.\end{equation}
Here $A\ll B$ means that there exists a constant $c>0$ such that $A\leq cB$, and $A\asymp B$ means that both $A\ll B$ and $B\ll A$ hold.

We introduce the following function spaces. Let $$C_{\sm}(X):=\bigcup_J C(X/J)$$ be the space of functions on $X$ that are right $J$-invariant for some compact open subgroup $J\subset G$. 
We denote by $\CC(X)$ (resp. $\CC_\temp(X)$) the space of Schwartz (resp. tempered) functions on $X$. These are the subspaces of $C_\sm(X)$ characterised by the growth condition  
\begin{equation}\label{equ estimate3} |f(x)|\ll\Xi^X(x)\N_{-d}(x),\quad\forall x\in X,\end{equation}
which is required to hold for every $d\in\BN$ in the Schwartz case, and for some $d\in\BZ$ in the tempered case. 
A fundamental fact is that if $\alpha\in\Hom_H(\pi,\BC)$ is relatively square-integrable (resp. relatively tempered), then the corresponding generalized matrix coefficient satisfies 
$\varphi_{\alpha,v}\in\CC(X)$ (resp. $\varphi_{\alpha,v}\in\CC_\temp(X)$).

\subsection{Base change doubling zeta integrals} In this subsection we
introduce the base change doubling zeta integral and discuss some of its elementary properties.

We work in two parallel situations. Let $\FG$ denote either $\Sp(\FW)$  or $\O(\FV_\tau)$, and accordingly
$G=\Sp(W)$ or $\O(V)$. Let $H=\Sp(W_F)$  (resp. $\O(V_F)$) and $X=H\bs G$ the corresponding Galois symmetric space. 

Let $\FP$ be the Siegel parabolic subgroup of $\FG$ stabilizing the Lagrangian subspace $\FX'$ (in the symplectic case) or $\tau V_F$ (in the orthogonal case); see Sect. \ref{sec deg} for the precise definition. 

Let $\pi\in\Irr(G)$.
Given a holomorphic section $\CF$ of  the degenerate principal series $I_\FP^\FG(s)$, a period $\alpha\in\Hom_{H}(\pi^\vee,\BC)$, and a vector $v\in\pi^\vee$, we define the \emph{base change doubling zeta integral} by
\begin{equation}\label{equ zeta}Z(s,\alpha,\CF,v):=\int_{X}\CF(g_\tau)\alpha(\pi^\vee(g)v)\ \d g,\end{equation} 
where $$g_\tau:=\begin{cases}\varphi_\tau(g)=c_\tau gc_\tau^{-1},\quad&\textrm{if } G=\Sp(W),
\\ g,\quad&\textrm{if }G=\O(V).\end{cases}$$ In the symplectic case, this integral admits an equivalent expression
$$Z(s,\alpha,\CF,v)=\int_{X^\tau}\CF(g)\alpha(\pi_\tau^\vee(g)v)\ \d g,$$ with $X^\tau:=\Sp(W_F)^\tau\bs\Sp(W)$.

This integral serves as a direct analogue of the classical doubling zeta integral introduced by Piatetski-Shapiro and Rallis \cite{gpsr}.

It is immediate that if $\alpha$ is relatively supercuspidal, then the integral $Z(s,\alpha,\CF,v)$ converges absolutely for every $s\in\BC$. For a fixed $s\in\BC$ for which these integrals are absolutely convergent, they define a $G$-invariant linear form
$$Z(s,\alpha,-,-):I_\FP^\FG(s)^\tau\otimes\pi^\vee\lra\BC,\quad \CF\otimes v\mapsto Z(s,\alpha,\CF,v),$$ where we set $I_\FP^\FG(s)^\tau:=I_\FP^\FG(s)\circ\varphi_\tau$.

The remainder of this subsection is devoted to proving the following proposition on the convergence of the base change doubling zeta integrals. In order to treat the symplectic and orthogonal cases simultaneously, we henceforth assume
$$\dim W=\dim V=2n.$$

\begin{prop}\label{prop zeta-conv}
\begin{enumerate}
	\item If $\alpha$ is relatively square-integrable, then $Z(s,\alpha,\CF,v)$ is absolutely convergent for $\Re(s)\geq-\frac{1}{2}$.
	\item If $\alpha$ is relatively tempered, then $Z(s,\alpha,\CF,v)$ is absolutely convergent for $\Re(s)>-\frac{1}{2}$.
\end{enumerate}
\end{prop}

\subsubsection{Maximal $\iota$-split tori}\label{subsec max-split} To study the integral (\ref{equ zeta}), we need to classify the maximal $\iota$-split tori of $G$, where $\iota$ denotes the Galois involution. Define the $n\times n$ anti-diagonal matrix
$$w_n:=\begin{pmatrix}&&1\\ &{ \begin{sideways}$\ddots$\end{sideways} }\\ 1&& \end{pmatrix}\in\GL_n.$$ Set $$J_n:=\begin{pmatrix}
0&w_n\\-w_n&0
\end{pmatrix},\quad \gamma_n:=\begin{pmatrix}
1&w_n\\-\tau w_n&\tau 
\end{pmatrix}.$$ 
 For any $k\leq n$, let $$A_k:=\left\{\left(\begin{smallmatrix}a_1&&&&& \\ &\ddots&&&& \\ &&a_k&&& \\ &&&a_k^{-1}&& \\ &&&&\ddots& \\ &&&&&a_1^{-1} \end{smallmatrix}\right) \mid a_1,...,a_k\in F^\times \right\}$$ and $$S_k:=\gamma_kA_k\gamma_k^{-1}.$$
We identify $A_k$ with $(F^\times)^k$ in the natural way. View $A_k$ as a subtorus of $A_n$ via the embedding $$(a_1,...,a_k)\mapsto (a_1,...,a_k,1,...,1),$$ and similarly regard $S_k$ as a subtorus of $S_n$ by the same identification.

For the Galois symmetric pair $(\GL_{2n}(E),\GL_{2n}(F))$, it is known that $S_n$ is a maximal $\iota$-split torus of $\GL_{2n}(E)$. Moreover, every $\iota$-split torus of $\GL_{2n}(E)$ is $\GL_{2n}(F)$-conjugate to $S_k$ for some $k\leq n$ (see \cite[Sect. 5.1]{smi}).

Let $G=\Sp(W)$. Suppose that the symplectic form on $W_F$ is represented by the matrix $J_n$. A direct verification shows that $S_n$ is contained in $G$; thus it is a maximal $\iota$-split torus of $G$. Moreover, one finds that $\RH^1(F,Z_G(S_n)\cap H)$ is trivial. Hence every maximal $\iota$-split torus of $G$ is $H$-conjugate to $S_n$.  For $a\in A_n$, we write
$$g(a):=\gamma_na\gamma_n^{-1}\in S_n.$$ 

Let $G=\O(V)$. Since any maximal $\iota$-split torus of $G$ is in particular an $\iota$-split torus of $\GL_{2n}(E)$, it must be $\GL_{2n}(F)$-conjugate to $S_r$, where $r$ denotes its split rank. If $S=\delta S_r\delta^{-1}$ with $\delta\in\GL_{2n}(F)$ is a maximal $\iota$-split torus of $G$, then for $a\in A_r$, we define
$$g(a):=\delta\gamma_ra\gamma_r^{-1}\delta^{-1}\in S.$$

\subsubsection{Asymptotic behavior of holomorphic sections} Set
$$A_k^+:=\left\{(a_1,...,a_k)\in(F^\times)^k\mid |a_1|_F\leq\cdots\leq|a_k|_F\leq 1\right\}.$$

\begin{lem}\label{lem asym-sec}
\begin{enumerate}
\item Let $G=\Sp(W)$. For $a\in A_n$, denote by $g(a)\in S_n$ the element constructed in Sect. \ref{subsec max-split}. Then for every $\CF\in I_\FP^\FG(s)$ and $a\in A_n^+$, 
$$|\CF(g(a)_\tau)|\ll\prod_{i=1}^n|a_i|_E^{\Re(s)+\rho},\quad \rho:=\frac{2n+1}{2}.$$ 
\item  Let $G=\O(V)$ and let $S$ be a maximal  $\iota$-split torus of $G$. For $a\in A_r$, write $g(a)\in S$ for the corresponding element. Then for any $\CF\in I_\FP^\FG(s)$ and $a\in A_r^+$, 
$$|\CF(g(a))|\ll\prod_{i=1}^r|a_i|_E^{\Re(s)+\rho},\quad \rho:=\frac{2n-1}{2}.$$ 
\end{enumerate}
\end{lem}

\begin{proof}
Following the proof of \cite[Lemma 9.4]{gi14}, it suffices to consider the specific section $\CF_0$ of $I_\FP^\FG(s)$ constructed in \cite[Sect. 6.4]{gpsr}.  We will then adapt the arguments of {\em loc. cit.} to compute  $\CF_0(g(a)_\tau)$ explicitly.

Suppose $G=\Sp(W)$. We express elements of $\FG=\Sp(\FW)$ in block form with respect to the decomposition $$\FW=W_F\oplus\tau W_F=X_F\oplus Y_F\oplus\tau X_F\oplus\tau Y_F.$$ For brevity, we also write $a=\diag(a_1,...,a_n)$ and $w=w_n$. Then
$$\begin{aligned}g(a)_\tau=c_\tau\gamma_na\gamma_n^{-1}c^{-1}_\tau&=\begin{pmatrix}1&w\\ -w&1\end{pmatrix}\begin{pmatrix}a&\\ &wa^{-1}w\end{pmatrix}\begin{pmatrix}1&w\\ -w&1\end{pmatrix}^{-1}\\
&=\frac{1}{2}\begin{pmatrix}1&w\\  -w&1\end{pmatrix}\begin{pmatrix}a&\\ &wa^{-1}w\end{pmatrix}\begin{pmatrix}1&-w\\  w&1\end{pmatrix}.
\end{aligned}$$ Rewriting this as a matrix in $\Sp(\FW)$ gives
$$g(a)_\tau=\frac{1}{2}\begin{pmatrix}1&w&&\\ -w&1&&\\ &&1&w\\ &&
-w&1
\end{pmatrix}\begin{pmatrix}
a&&&\\ &wa^{-1}w&&\\ &&a&\\ &&&wa^{-1}w
\end{pmatrix}\begin{pmatrix}
1&-w&&\\ w&1&&\\ &&1&-w\\ &&w&1\end{pmatrix}.$$

Let $\Phi_0$ be the characteristic function of $\M_{2n\times 4n}(\CO_F)$, where $\M_{2n\times 4n}$ denotes the set of $2n\times 4n$ matrices. Define the section $\CF_0$ by
$$\CF_0(g):=\xi(s)^{-1}\int_{\GL_{2n}(F)}\Phi_0((Z_1,0,0,Z_2)g)|\det Z|_F^s\ \d Z,$$ where $Z\in\GL_{2n}(F)$ is written in block form as $Z=(Z_1,Z_2)$ with $Z_i\in\M_{2n\times n}(F)$, and 
$$\xi(s):=\int_{\GL_{2n}(F)}\Phi_0((Z_1,0,0,Z_2))|\det Z|_F^s\ \d Z.$$
Then $\CF_0$ belongs to the degenerate principal series
$ I_\FP^\FG(s-\rho)$. 

By a direct computation, we obtain
$$\begin{aligned}&(Z_1,0,0,Z_2)\cdot g(a)_\tau\\ =&Z\cdot\underbrace{\begin{pmatrix}
1&0\\0&w
\end{pmatrix}\begin{pmatrix}
\dfrac{a+a^{-1}}{2}&\dfrac{-a+a^{-1}}{2}&&\\&&\dfrac{-a+a^{-1}}{2}&\dfrac{a+a^{-1}}{2}\end{pmatrix}\begin{pmatrix}
1&&&\\&w&&\\&&1&\\&&&w
\end{pmatrix}}_{h(a):=}.\end{aligned}$$
This motivates us to introduce the auxiliary function
$$\kappa(h):=\xi(s)^{-1}\int_{\GL_{2n}(F)}\Phi_0(Zh)|\det Z|^s\ \d Z,\quad  h\in\M_{2n\times 4n}(F).$$ Observe that
$$\CF_0(g(a)_\tau)=\kappa(h(a)).$$

By the invariance and homogeneity properties of  $\kappa$ (see \cite[Lemma 6.8]{gpsr}), we have 
$$\kappa(h)=\kappa'(h)^{-s},$$ where
$\kappa'(h):=\max_{M\in\CM(h)}(|\det M|)$ and $\CM(h)$ denotes the set of all $2n\times 2n$ minors of $h$. 
A straightforward computation shows that 
$$|\kappa'(h(a))|\gg\prod_{i=1}^n|a_i|_F^{-2}.$$ Therefore,
$$|\CF_0(g(a)_\tau)|\ll\prod_{i=1}^n|a_i|_E^{\Re(s)}.$$ 

The above proof, after a straightforward adaptation, also applies to the case of $G=\O(V)$. Suppose $G=\O(V)$ and let $S=\delta\gamma_rA_r\gamma_r^{-1}\delta^{-1}$ with $\delta\in\GL_{2n}(F)$. We write elements of $\FG=\O(\FV_\tau)$ in block form with respect to the decomposition
$$\FV_\tau=\tau V_F\oplus V_F.$$ Note that an element $g\in\GL_{2n}(E)$ can be written as $g=A+\tau B$ with $A, B\in\M_{2n}(F)$. When viewed as a matrix in $\GL_{4n}(F)$ via restriction of scalars, it admits the block form
$$g=\begin{pmatrix}A&\beta B\\ B&A\end{pmatrix},$$ where $\beta=\tau^2\in F^\times$.

Define $$\CF_0(g):=\xi(s)^{-1}\int_{\GL_{2n}(F)}\Phi_0((Z,0)g)|\det Z|^s\ \d Z,$$ where
$$\xi(s):=\int_{\GL_{2n}(F)}\Phi_0((Z,0))|\det Z|^s\ \d Z.$$
Then $\CF_0$ belongs to the degenerate principal series $I_\FP^\FG(s-\rho)$. 

For $g=A+\tau B\in\O(V)$, we have
$$(Z,0)g=Z(A,\beta B).$$ In particular, for 
$$g(a)=\delta\gamma_ra\gamma_r^{-1}\delta^{-1}\in S,$$ a direct computation yields
$$\begin{aligned}&(Z,0)\cdot g(a)\\
&=Z\cdot\underbrace{\delta\begin{pmatrix}1&\\&w\end{pmatrix}\begin{pmatrix}\dfrac{a+a^{-1}}{2}&&&\dfrac{-a+a^{-1}}{2}\\ &\dfrac{a+a^{-1}}{2} &\dfrac{\beta(-a+a^{-1})}{2}& \end{pmatrix}\begin{pmatrix}1&&&\\&w&&\\ &&1&\\ &&&w\end{pmatrix}\begin{pmatrix}\delta^{-1}&\\ &\delta^{-1}\end{pmatrix}}_{h(a):=}.\end{aligned}$$ Consequently,
$$\CF_0(g(a))=\kappa(h(a))=\kappa'(h(a))^{-s}.$$ 
A straightforward estimate then gives
$$|\CF_0(g(a))|\ll\prod_{i=1}^r|a_i|_E^{\Re(s)}.$$
\end{proof}

\subsubsection{Completion of the proof} We now complete the proof of Proposition \ref{prop zeta-conv}.

\begin{proof}[Proof of Proposition \ref{prop zeta-conv}]
We prove this proposition for $G=\Sp(W)$; the orthogonal case follows by the same argument, with the obvious notational changes.

Let $P$ be a minimal $\iota$-parabolic subgroup of $G$ and set $S:=S_P$.
By the weak Cartan decomposition (\ref{equ cartan}) and the volume estimate (\ref{equ estimate2}), it suffices to study the integral
$$\int_{S^+}\CF(a_\tau)\varphi_{\alpha,v}(a)\delta_P^{-1}(a)\ \d a.$$
Identify $S$ with $(F^\times)^n$.  Then  
$$\delta_{P}(a)=\prod_{i=1}^n|a_i|_E^{2\rho+1-2i},\quad  a=(a_1,...,a_n)\in S.$$ Assume first that $\alpha$ is relatively square-integrable. By (\ref{equ estimate1}) and (\ref{equ estimate3}),  for any $d\in\BN$ we have $$|\varphi_{\alpha,v}(a)|\ll\delta_P(a)^{1/2}(a)\N_{-d}(a).$$ Applying Lemma \ref{lem asym-sec}, we obtain
$$\begin{aligned}\int_{S^+}|\CF(a_\tau)\varphi_{\alpha,v}(a)|\delta_P^{-1}(a)\ \d a&\ll\int_{S^+}\prod_{i=1}^r|a_i|_E^{\Re(s)+\rho}\cdot\delta_{P}^{1/2}(a)\N_{-d}(a)\cdot\delta_{P}^{-1}(a)\ \d a\\
&=\int_{S^+}\prod_{i=1}^n|a_i|_E^{\Re(s)+i-1/2}\N_{-d}(a)\ \d a.\end{aligned}$$ A standard estimate shows that this integral converges absolutely for $\Re(s)\geq-\frac{1}{2}$ provided $d$ is chosen sufficiently large. The case where $\alpha$ is relatively tempered is handled similarly, using the corresponding estimate in \eqref{equ estimate3} for relatively tempered periods.
\end{proof}

\subsection{Test functions} From now on, we return to the general setting $$\dim W_F=2n,\quad \dim V_F=2m,$$ and resume the notation in Sects. \ref{sec deg} and \ref{sec mult}.
We denote the relevant Galois symmetric spaces by $$X_{W_F}:=\Sp(W_F)\bs\Sp(W),\quad X_{V_F}:=\O(V_F)\bs\O(V).$$

Recall that the Weil representation $\omega_{V,W,\psi}$ of $\O(V)\times\Sp(W)$ is realized on the Schwartz space $C_c^\infty(\BY)$, while the Weil representation $\omega_{V_F,\FW,\psi_F}$ of $\O(V_F)\times\Sp(\FW)$ is realized on the Schwartz space $C_c^\infty(\BY^\square)$. We identify $C_c^\infty(\BY^\square)$ with $C_c^\infty(\BY)$. Furthermore, $\omega_{V_F,\FW,\psi_F}$ admits another realization on the Schwartz space $C_c^\infty(\BY^{\square'})$, and we have the intertwining operator 
$$\CI_\FW: C_c^\infty(\BY)=C_c^\infty(\BY^\square)\lra C_c^\infty(\BY^{\square'}),$$ as defined in \eqref{equ int-1}.
By composing $R_\FW$ with $\CI_\FW$ and applying the $\tau$-twist, we obtain a linear map
\begin{equation}\label{equ test-map}p_{V_F}:C_c^\infty(\BY)\lra C_\sm\left(X_{W_F}\right),\end{equation} given explicitly by
$$\begin{aligned}p_{V_F}(\phi)(g)&=\omega_{V_F,\FW,\psi_F}(g_\tau)(\CI_\FW(\phi))(0)\\
&=\CI_\FW(\omega_{V,W,\psi}(g_\tau)\phi)(0)\\
&=\CI_\FW(\omega_{V,W,\psi_\tau}(g)\phi)(0).\end{aligned}$$
In a completely analogous fashion, using the Weil representations $\omega_{V_\rc,W,\psi}$ and $\omega_{V_{\rc,F},\FW,\psi_F}$, we define a map
$$p_{V_{\rc,F}}:C_c^\infty(\BY_\rc)\lra C_\sm(X_{W_F}),$$ where $\BY_\rc=Y\otimes_EV_\rc$.  

We also have the Weil representations $\omega_{V,W,\psi_\tau}$ of $\O(V)\times\Sp(W)$ and $\omega_{\FV_\tau,W_F,\psi_F}$ of $\O(\FV_\tau)\times\Sp(W_F)$; both are realized on the common Schwartz space $C_c^\infty(\BY)=C_c^\infty(\BY_\tau)=C_c^\infty(\BY_\tau^\square)$. By composing the intertwining operator (\ref{equ int-2})
$$\CI_{\FV_\tau}:C_c^\infty(\BY_\tau^\square)\lra C_c^\infty(\BY_\tau^{\square'})$$ with the Rallis map $R_{\FV_\tau}$, we obtain a linear map
\begin{equation}\label{equ test-map2}
q_{W_F}:C_c^\infty(\BY)\lra C_\sm\left(X_{V_F}\right),\end{equation} given explicitly by
$$q_{W_F}(\phi)(h)=\omega_{\FV_\tau,W_F,\psi_F}(h)(\CI_{\FV_\tau}(\phi))(0)=\CI_{\FV_\tau}(\omega_{V_\tau,W,\psi}(h)\phi)(0).$$ 

\begin{lem}\label{lem test-fcn-com}
For any $\phi\in C_c^\infty(\BY)$, $g\in\Sp(W)$, and $h\in\O(V)$, we have
$$p_{V_F}(\omega_{V,W,\psi_\tau}(h)\phi)(g)=q_{W_F}(\omega_{V,W,\psi_\tau}(g)\phi)(h).$$ 
\end{lem}

\begin{proof}
Using the definitions of $p_{V_F}$ and $q_{W_F}$ together with Lemma \ref{lem int-val1}, we have
$$\begin{aligned}
p_{V_F}(\omega_{V,W,\psi_\tau}(h)\phi)(g)
&=\CI_{\FW}(\omega_{V,W,\psi_\tau}(g)\omega_{V,W,\psi_\tau}(h)\phi)(0)\\
&=\CI_{\FV_\tau}(\omega_{V,W,\psi_\tau}(h)\omega_{V,W,\psi_\tau}(g)\phi)(0)\\
&=q_{W_F}(\omega_{V,W,\psi_\tau}(g)\phi)(h).
\end{aligned}$$
\end{proof}

\begin{defn}\label{def function}\begin{enumerate}
\item Let $\CS\left(X_{W_F}\right)$ be the image of the map 
$$p_{V_F}\oplus p_{V_{\rc,F}}: C_c^\infty(\BY)\oplus C_c^\infty(\BY_\rc)\lra C_\sm\left(X_{W_F}\right),$$ and let $\CS\left(X_{V_F}\right)$ be the image of $p_{W_F}$. We call $\CS\left(X_{W_F}\right)$ and 
$\CS\left(X_{V_F}\right)$ the \emph{spaces of test functions}.
\item For $f\in C_\sm(X_{W_F})$ and $f'\in C_\sm(X_{V_F})$, we say that $f$ and $f'$ \emph{correspond to each other} if there exists $\phi\in C_c^\infty(\BY)$ such that $$f=p_{V_F}(\phi),\quad f'=q_{W_F}(\phi).$$
\end{enumerate}
\end{defn}

\begin{rem}\label{rem fcn-deg}
From the definition we obtain the following descriptions. For the symplectic side, $$\CS(X_{W_F})=\left(R_\FW(V_F)+R_\FW(V_{\rc,F})\right)^\tau|_{\Sp(W)},$$ where the superscript $\tau$ denotes composition with $\varphi_\tau$. For the orthogonal side, we have
$$\CS(X_{V_F})=R_{\FV_\tau}(W_F)|_{\O(V)}=(R_{\FV_\tau}(W_F)+R_{\FV_\tau}(W_F)\otimes\det)|_{\O(V)},$$ since $R_{\FV_\tau}(W_F)=R_{\FV_\tau}(W_F)\otimes\det$.
\end{rem}

\begin{prop}\label{prop test-analytic}
If $m>n$, then
$$C_c^\infty\left(X_{W_F}\right)\subset \CS\left(X_{W_F}\right)\subset \CC\left(X_{W_F}\right).$$ If $m\leq n$, then
$$C_c^\infty\left(X_{V_F}\right)\subset \CS\left(X_{V_F}\right)\subset \CC\left(X_{V_F}\right).$$
\end{prop}

\begin{proof}
From Remark \ref{rem fcn-deg}, Proposition \ref{prop degps-symp} and Proposition \ref{prop degps-orth} we obtain $$\CS(X_{W_F})=I_\FP^\FG(s_0)^\tau|_{\Sp(W)},\quad \CS(X_{V_F})=I_\FQ^\FH(-s_0)|_{\O(V)}.$$ Note that 
$$C_c^\infty(X_{W_F}^\tau)^\tau= C_c^\infty(X_{W_F}).$$
By Proposition \ref{prop degps-symp-2} and Proposition \ref{prop degps-orth-2} we then have
 $$C_c^\infty(X_{W_F})\subset\CS(X_{W_F}),\quad C_c^\infty(X_{V_F})\subset\CS(X_{V_F}).$$ The asymptotic behavior of the degenerate principal series (Lemma \ref{lem asym-sec}) further gives
 $$\CS(X_{W_F})\subset \CC(X_{W_F}),\quad \CS(X_{V_F})\subset \CC(X_{V_F}).$$ 
\end{proof}

\begin{rem}\label{rem test-analytic}
Furthermore, Lemma \ref{lem asym-sec} implies the following temperedness properties: $$\CS\left(X_{W_F}\right)\subset\CC_\temp\left(X_{W_F}\right)\quad\textrm{for}\ \ m=n,$$ and
$$\CS\left(X_{V_F}\right)\subset\CC_\temp\left(X_{V_F}\right)\quad\textrm{for}\ \ m=n+1.$$
\end{rem}

\subsection{Transfer of periods}
Let $\pi\in\Irr\left(\Sp(W)\right)$. For $\alpha\in\Hom_{\Sp(W_F)}\left(\pi^\vee,\BC\right)$, we formally define a linear form $$T_{W_F}(\alpha):\omega_{V,W,\psi_\tau}\otimes\pi^\vee\lra\BC$$ by
\begin{equation}\label{equ def-transf Sp}
\begin{aligned}T_{W_F}(\alpha)(\phi,v)&=\int_{X_{W_F}}p_{V_F}(\phi)(g)\alpha\left(\pi^\vee(g)v\right)\ \d g\\
&=\int_{X_{W_F}}\CI_\FW(\omega_{V,W,\psi_\tau}(g)\phi)(0)\alpha\left(\pi^\vee(g)v\right)\ \d g.\end{aligned}\end{equation} Similarly, for $\sigma\in\Irr\left(\O(V)\right)$ and $\beta\in\Hom_{\O(V_F)}\left(\sigma^\vee,\BC\right)$, we formally define  $$T_{V_F}(\beta):\omega_{V,W,\psi_\tau}\otimes\sigma^\vee\lra\BC$$ by
\begin{equation}\label{equ def-transf O}T_{V_F}(\beta)(\phi,w)=\int_{X_{V_F}}q_{W_F}(\phi)(h)\beta\left(\sigma^\vee(h)w\right)\ \d h.\end{equation}
The linear forms $T_{W_F}(\alpha)$ and $T_{V_F}(\beta)$ are well defined whenever the integrals in (\ref{equ def-transf Sp}) and (\ref{equ def-transf O}) converge absolutely.

\begin{lem}\label{lem well-def-transf} 
\begin{enumerate}
\item The linear form $T_{W_F}(\alpha)$ is well defined in the following cases:
\begin{itemize}
\item $\alpha$ is relatively supercuspidal;
\item $\alpha$ is relatively square-integrable and $m\geq n$;
\item $\alpha$ is relatively tempered and $m>n$.
\end{itemize}
\item The linear form $T_{V_F}(\beta)$ is well defined in the following cases:
\begin{itemize}
\item $\beta$ is relatively supercuspidal;
\item $\beta$ is relatively square-integrable and $m\leq n+1$;
\item $\beta$ is relatively tempered and $m\leq n$.
\end{itemize}
\end{enumerate}
\end{lem}

\begin{proof}
This is immediate from Proposition \ref{prop test-analytic} and Remark \ref{rem test-analytic}.
\end{proof}

From now on, we assume that $T_{W_F}(\alpha)$ and $T_{V_F}(\beta)$ are well defined.
It is clear that $T_{W_F}(\alpha)$ is $\Sp(W)$-invariant and $T_{V_F}(\beta)$ is $\O(V)$-invariant. Hence, by the definition of the big theta lift together with the base change seesaw identity, $T_{W_F}(\alpha)$ descends to a linear form
$$\Theta_{V,W,\psi_\tau}(\pi)\lra\BC,$$ which is  $\O(V_F)$-invariant. We continue to denote this induced map by $T_{W_F}(\alpha)$; thus 
\begin{equation}\label{equ period-symp}T_{W_F}(\alpha)\in\Hom_{\O(V_F)}\left(\Theta_{V,W,\psi_\tau}(\pi),\BC\right).\end{equation} More concretely, let $\theta(\phi,v)$ denote the image of $\phi\otimes v$ under the canonical surjection
$$\omega_{V,W,\psi_\tau}\otimes\pi^\vee\lra\Theta_{V,W,\psi_\tau}(\pi).$$ Then $T_{W_F}(\alpha)$ is characterized by $$T_{W_F}(\alpha)(\theta(\phi,v)):=T_{W_F}(\alpha)(\phi,v).$$
In an entirely analogous manner, we obtain a local period \begin{equation}\label{equ period-orth}T_{V_F}(\beta)\in\Hom_{\Sp(W_F)}\left(\Theta_{V,W,\psi_\tau}(\sigma),\BC\right).\end{equation}

According to Lemma \ref{lem well-def-transf}, the linear map
 \begin{equation}\label{equ period-symp2}T_{W_F}:\Hom_{\Sp(W_F)}\left(\pi^\vee,\BC\right)\lra\Hom_{\O(V_F)}\left(\Theta_{V,W,\psi_\tau}(\pi),\BC\right),\quad\alpha\mapsto T_{W_F}(\alpha),\end{equation} is well defined in the following cases:
 \begin{itemize}
\item $\pi$ is  $\Sp(W_F)$-relatively supercuspidal;
\item $\pi$ is  $\Sp(W_F)$-relatively square-integrable and $m\geq n$;
\item $\pi$ is $\Sp(W_F)$-relatively tempered and $m>n$.
\end{itemize} 
Similarly, we have a well defined linear map
\begin{equation}\label{equ period-orth2}T_{V_F}:\Hom_{\O(V_F)}\left(\sigma^\vee,\BC\right)\lra\Hom_{\Sp(W_F)}\left(\Theta_{V,W,\psi_\tau}(\sigma),\BC\right),\quad \beta\mapsto T_{V_F}(\beta),\end{equation} in the following cases:
 \begin{itemize}
\item $\sigma$ is $\O(V_F)$-relatively supercuspidal;
\item $\sigma$ is $\O(V_F)$-relatively square-integrable and $m\leq n+1$;
\item $\sigma$ is $\O(V_F)$-relatively tempered and $m\leq n$.
\end{itemize}

As $V'_F$ ranges over $\RH^1(E/F,\O(V))$, we may combine the maps $T_{V'_F}$ to obtain a linear map $T_V:=\bigoplus_{V'_F\in\RH^1(E/F,\O(V))}T_{V'_F}: $
$$\bigoplus_{V'_F\in\RH^1(E/F,\O(V))}\Hom_{\O(V'_F)}\left(\sigma^\vee,\BC\right)\lra\Hom_{\Sp(W_F)}\left(\Theta_{V,W,\psi_\tau}(\sigma),\BC\right).$$

\begin{thm}\label{thm transf1}
Let $\pi\in\Irr(\Sp(W))$.
Assume that $m>n$ and that $\Theta_{V,W,\psi_\tau}(\pi)$ is the first occurrence. 
\begin{enumerate}
\item If $\pi$ is tempered,
then $T_{W_F}$ is injective.
\item If $\pi$ is square-integrable and $m=n+1$, then $T_{W_F}$ is an isomorphism.
\item If $\pi$ is supercuspidal, then $T_{W_F}$ is an isomorphism.
\end{enumerate}
\end{thm}

\begin{proof} Let $\alpha$ be a non-zero element of $\Hom_{\Sp(W_F)}\left(\pi^\vee,\BC\right)$.
The Frobenius reciprocity gives a natural isomorphism $$\Hom_{\Sp(W_F)}\left(\pi^\vee,\BC\right)\stackrel{\simeq}{\lra}\Hom_{\Sp(W)}\left(C_c^\infty(X_{W_F})\otimes\pi^\vee,\BC\right),$$ which sends $\alpha$ to the linear form $\ell(\alpha)$ defined explicitly by
$$\ell(\alpha)(f,v)=\int_{X_{W_F}}f(g)\alpha(\pi^\vee(g)v)\ \d g.$$ By Proposition \ref{prop zeta-conv}, the hypotheses on $\pi$ guarantee  that $\ell(\alpha)$ extends to a non-zero element $$Z(s_0,\alpha,\textrm{-},\textrm{-})\in\Hom_{\Sp(W)}\left(I_\FP^\FG(s_0)^\tau\otimes\pi^\vee,\BC\right).$$ It is then immediate that $$T_{W_F}(\alpha)=Z(s_0,\alpha,\textrm{-},\textrm{-})|_{R_{\FW}(V_F)^\tau\otimes\pi^\vee}.$$ From the injection (\ref{equ inj1}) we deduce that $T_{W_F}(\alpha)$ is non-zero. Hence $T_{W_F}$ is injective. 

Under the assumptions of parts (2) and (3), it follows from the discussion above that $T_{W_F}$ factors as the composition of the following isomorphisms:
$$\begin{aligned}&\Hom_{\Sp(W_F)}\left(\pi^\vee,\BC\right)\stackrel{\simeq}{\lra}\Hom_{\Sp(W)}\left(C_c^\infty(X_{W_F})\otimes\pi^\vee,\BC\right)\stackrel{\simeq}{\lra}\Hom_{\Sp(W)}\left(I_\FP^\FG(s_0)^\tau\otimes\pi^\vee,\BC\right)\\
\stackrel{\simeq}{\lra}&\Hom_{\Sp(W)}(R_\FW(V_F)^\tau\otimes\pi^\vee,\BC)\stackrel{\simeq}{\lra}\Hom_{\Sp(W)}(R_\FW(V_F),\pi_\tau)\stackrel{\simeq}{\lra}\Hom_{\O(V_F)}\left(\Theta_{V,W,\psi_\tau}(\pi),\BC\right).\end{aligned}$$ Consequently, $T_{W_F}$ is an isomorphism.
\end{proof} 

\begin{thm}\label{thm transf2}
Let $\sigma\in\Irr(\O(V))$. Assume that $m\leq n$ and that $\Theta_{V,W,\psi_\tau}(\sigma)$ is the first occurrence. 
\begin{enumerate}
\item If $\sigma$ is tempered, then 
$T_V$ is injective.
\item If $\sigma$ is square-integrable and $m=n$, then
$T_V$ is an isomorphism.
\item If $\sigma$ is supercuspidal, then $T_V$ is an isomorphism.
\end{enumerate}
\end{thm}
\begin{proof}
The proof follows the same lines as that of Theorem \ref{thm transf1}, with the obvious modifications. We omit the details.
\end{proof}

\subsection{Adjoint relation}\label{subsec adj rel}
Suppose that $\pi\in\Irr\left(\Sp(W)\right)$ is unitary. Then we have a natural identification $\pi^\vee=\overline{\pi}$, where $\overline{\pi}$ denotes the complex conjugate representation of $\pi$. For a linear form $\alpha:\pi\ra\BC$, we define a linear form $\overline{\alpha}:\pi^\vee\ra\BC$  by $\overline{\alpha}(v)=\overline{\alpha(v)}$. This yields an isomorphism $$\Hom_{\Sp(W_F)}\left(\pi,\BC\right)\stackrel{\simeq}{\lra}\Hom_{\Sp(W_F)}\left(\pi^\vee,\BC\right),\quad \alpha\mapsto\overline{\alpha}.$$ Consequently, whenever $T_{W_F}$ is well defined on $\Hom_{\Sp(W_F)}\left(\pi^\vee,\BC\right)$, we obtain a linear map
 $$\Hom_{\Sp(W_F)}\left(\pi,\BC\right)\lra\Hom_{\O(V_F)}\left(\Theta_{V,W,\psi_\tau}(\pi),\BC\right),\quad \alpha\mapsto T_{W_F}(\overline{\alpha}),$$ which we continue to denote  by $T_{W_F}$. 
 
 Similarly, if $\sigma\in\Irr(\O(V))$ is unitary, we  obtain a linear map
 $$T_{V_F}:\Hom_{\O(V_F)}\left(\sigma,\BC\right)\lra\Hom_{\Sp(W_F)}\left(\Theta_{V,W,\psi_\tau}(\sigma),\BC\right),$$ defined in the same formal manner as above.

Assume that $\pi\in\Irr\left(\Sp(W)\right)$ is square-integrable. In particular, $\pi$ is unitary and $\Sp(W_F)$-relatively square-integrable. There exists an inner product  $\pair{\textrm{-},\textrm{-}}_{X_{W_F}}$ on the space $\Hom_{\Sp(W_F)}\left(\pi,\BC\right)$, uniquely characterized by the identity
$$\int_{X_{W_F}}\alpha(\pi(g)v)\overline{\alpha'(\pi(g)v')}\ \d g=\pair{\alpha,\alpha'}_{X_{W_F}}\cdot\pair{v,v'}_\pi,$$ for all $\alpha,\alpha'\in \Hom_{\Sp(W_F)}\left(\pi,\BC\right)$ and all $v,v'\in\pi,$ where $\pair{\textrm{-},\textrm{-}}_\pi$ is a fixed inner product on $\pi$. Note that
$\pair{\textrm{-},\textrm{-}}_{X_{W_F}}$ depends on the choice of the Haar measure $\d g$ on $X_{W_F}$ and the inner product $\pair{\textrm{-},\textrm{-}}_\pi$.

Similarly, we have an inner product $\pair{\textrm{-},\textrm{-}}_{X_{V_F}}$ on $\Hom_{\O(V_F)}(\sigma,\BC)$ for  irreducible  square-integrable representation $\sigma$ of $\O(V)$.

\begin{thm}\label{thm adjoint} Suppose that $\pi\in\Irr(\Sp(W))$ is supercuspidal and that $\sigma=\Theta_{V,W,\psi_\tau}(\pi)$ is the first occurrence. Then there exists a non-zero constant  $c\in\BC^\times$ such that
$$\pair{T_{W_F}(\alpha),\beta}_{X_{V_F}}=c\pair{T_{V_F}(\beta),\alpha}_{X_{W_F}}$$ for all $\alpha\in\Hom_{\Sp(W_F)}\left(\pi,\BC\right)$ and $\beta\in\Hom_{\O(V_F)}\left(\sigma,\BC\right)$.
\end{thm}

\begin{proof}
Observe that $\sigma$ is irreducible and supercuspidal, and by the Howe duality we have $\pi=\Theta_{V,W,\psi_\tau}(\sigma)$. 

As before, for $\phi\in C_c^\infty(\BY)$ and $v\in\pi$, we let $\theta(\phi,v)$ denote the image of $\phi\otimes v\in\omega_{V,W,\psi_\tau}\otimes\overline{\pi}$ under the natural projection onto $\sigma=\Theta_{V,W,\psi_\tau}(\pi)$. In the same way, for $w\in\sigma$ we obtain an element $\theta(\phi,w)\in\pi$. According to \cite[Equation (5.9)]{gw21}, there exists a constant $c\in\BC^\times$ such that \begin{equation}\label{equ inner-equal}\pair{\theta(\phi,v),w}_\sigma=c\pair{\theta(\phi,w),v}_\pi\end{equation} holds for every $v\in\pi$, $w\in\sigma$, and $\phi\in C_c^\infty(\BY)$.

By definition, we have 
\begin{equation}\label{equ adjoint}\begin{aligned}
&\pair{T_{V_F}(\beta),\alpha}_{X_{W_F}}\cdot\pair{\theta(\phi,w),v}_\pi\\
=&\int_{X_{W_F}}T_{V_F}(\beta)\left(\pi(g)\theta(\phi,w)\right)\overline{\alpha(\pi(g)v)}\ \d g\\
=&\int_{X_{W_F}}\left(\int_{X_{V_F}} q_{W_F}\left(\omega_{V,W,\psi_\tau}(g)\phi\right)(h)\overline{\beta(\sigma(h)w)}\ \d h\right)\overline{\alpha(\pi(g)v)}\ \d g.
\end{aligned}\end{equation}

Since $\pi$ and $\sigma$ are supercuspidal, the local periods $\alpha$ and $\beta$ are relatively supercuspidal, and thus we may interchange the order of integration. Applying Lemma \ref{lem test-fcn-com}, the last expression in (\ref{equ adjoint}) is equal to
$$\begin{aligned}
&\int_{X_{V_F}}\left(\int_{X_{W_F}}p_{V_F}(\omega_{V,W,\psi_\tau}(h)\phi)(g)\overline{\alpha(\pi(g)v)}\ \d g\right)\overline{\beta(\sigma(h)w)}\ \d h\\
=&\int_{X_{V_F}}T_{W_F}(\alpha)\left(\sigma(h)\theta(\phi,v)\right)\overline{\beta(\sigma(h)w)}\ \d h\\
=&\pair{T_{W_F}(\alpha),\beta}_{X_{V_F}}\cdot\pair{\theta(\phi,v),w}_\sigma.
\end{aligned}$$ 
The desired identity now follows from (\ref{equ inner-equal}).
\end{proof}

\begin{rem}\label{rem c and abs}\begin{enumerate}
\item In fact, we can choose inner products $\pair{\textrm{-},\textrm{-}}_\pi$ and $\pair{\textrm{-},\textrm{-}}_\sigma$ such that $$c=1$$ in (\ref{equ inner-equal}). From now on, we shall always make such choices.
\item If $m=n$ or $n+1$, $\pi$ is square-integrable, and $\sigma=\Theta_{V,W,\psi_\tau}(\pi)$ is irreducible and square-integrable, we expect that the last integral in (\ref{equ adjoint}) is absolutely convergent. If this is the case, the assertion of  Theorem \ref{thm adjoint} also holds in this situation.
\end{enumerate}
\end{rem}

If both $\pi$ and $\sigma=\Theta_{V,W,\psi_\tau}(\pi)$ are irreducible and unitary, and if the maps $T_{W_F}$ and $T_{V_F}$ are well defined, we obtain linear endomorphisms
\begin{equation}\label{equ endo-sp}
\begin{aligned}
L_{W_F}&:=T_{V_F}\circ T_{W_F}\in\End\left(\Hom_{\Sp(W_F)}(\pi,\BC)\right),\\
L_{V_F}&:=T_{W_F}\circ T_{V_F}\in\End\left(\Hom_{\O(V_F)}(\sigma,\BC)\right).
\end{aligned}
\end{equation}

\begin{cor}\label{cor normal} Suppose that $\pi\in\Irr(\Sp(W))$ is supercuspidal and that $\sigma=\Theta_{V,W,\psi_\tau}(\pi)$ is the first occurrence.
Then $L_{W_F}$ and $L_{V_F}$ are normal operators.
\end{cor}

\begin{proof}
For any $\alpha,\alpha'\in\Hom_{\Sp(W_F)}(\pi,\BC)$, Theorem \ref{thm adjoint} together with the choice of inner products as in Remark \ref{rem c and abs} (so that $c=1$) yields
$$\begin{aligned}
\pair{L_{W_F}(\alpha),\alpha'}_{X_{W_F}}&=\pair{T_{W_F}(\alpha'),T_{W_F}(\alpha)}_{X_{W_F}}\\
&=\overline{\pair{T_{W_F}(\alpha),T_{W_F}(\alpha')}_{X_{W_F}}}\\
&=\overline{\pair{L_{W_F}(\alpha'),\alpha}_{X_{W_F}}}\\
&=\pair{\alpha,L_{W_F}(\alpha')}_{X_{W_F}}.
\end{aligned}$$ 
Hence $L_{W_F}$ is a normal operator. The same argument, applied to $L_{V_F}$, shows that it is also normal.
\end{proof}

\begin{rem}
\begin{enumerate}
\item Corollary \ref{cor normal} is a direct consequence of Theorem \ref{thm adjoint}. We conjecture that it remains valid under the conditions of part (2) of Remark \ref{rem c and abs}.
\item If $L_{W_F}$ is a normal operator, we may decompose the space $\Hom_{\Sp(W_F)}(\pi,\BC)$ into eigenspaces:
$$\Hom_{\Sp(W_F)}(\pi,\BC)=\bigoplus_\lambda\Hom_{\Sp(W_F)}(\pi,\BC)_\lambda,$$ where $\lambda$ runs over the eigenvalues of $L_{W_F}$ and the subscript denotes the corresponding eigenspace.  A natural question is to study the properties of these eigenvalues $\lambda$ and the representation-theoretic meaning of the dimensions $\dim\Hom_{\Sp(W_F)}(\pi,\BC)_\lambda$.
\end{enumerate}
\end{rem}

\begin{cor}\label{cor normal-isom}
Suppose that $T_{W_F}$ and $T_{V_F}$ are well defined and satisfy the adjoint relation of Theorem \ref{thm adjoint}.  If $T_{W_F}$ (resp. $T_{V_F}$) is injective, then $L_{W_F}$ (resp. $L_{V_F}$) is an automorphism. 
\end{cor}

\begin{proof}
Assume that $T_{W_F}$ is injective. For any non-zero $\alpha\in\Hom_{\Sp(W_F)}(\pi,\BC)$, we have
$$\pair{L_{W_F}(\alpha),\alpha}_{X_{W_F}}=\pair{T_{W_F}(\alpha),T_{W_F}(\alpha)}_{X_{W_F}}\neq 0.$$ Thus $L_{W_F}$ is injective, and consequently an automorphism. The same reasoning, applied to $T_{V_F}$, shows that $L_{V_F}$ is also an automorphism whenever $T_{V_F}$ is injective.
\end{proof}

\subsection{Relative character relation}\label{subsec rel-char}
We begin by recalling the definition of relative characters.

Let $X=H\bs G$ be a symmetric space and $\pi\in\Irr(G)$ a unitary representation. As usual, we identify the contragredient $\pi^\vee$ with the complex conjugate $\overline{\pi}$.
For $f\in C_c^\infty(X)$ and $\alpha\in\Hom_{H}(\pi,\BC)$, we define an element 
$\pi(f)\alpha\in\overline{\pi}^*$ by
\begin{equation}\label{equ dual-vector}\left(\pi(f)\alpha\right)(v)=\int_{X}f(g)\overline{\alpha(\pi(g)v)}\ \d g,\quad v\in\overline{\pi},\end{equation} where $\overline{\pi}^*$ denotes the linear dual of $\overline{\pi}$. Moreover, since $f$ is right $J$-invariant for some compact open subgroup $J\subset G$, one has
\begin{equation}\label{equ smooth-vector}\pi(f)\alpha\in(\overline{\pi}^*)^J\subset\pi.\end{equation} 

\begin{defn} For a unitary representation $\pi\in\Irr(G)$ and local periods $\alpha, \alpha'\in\Hom_{H}(\pi,\BC)$, we define the
 \emph{relative character} of $\pi$ associated to $\alpha, \alpha'$ as the distribution
\begin{equation}\label{def rel-char}\Phi_{\pi,\alpha,\alpha'}:C_c^\infty(X)\lra\BC,\quad f\mapsto\alpha'(\pi(f)\alpha).\end{equation} 
\end{defn}

Let $\ONB(\pi)$ be an orthonormal basis of $\pi$. Then the relative character can be expressed as
\begin{equation}\label{equ rel-char-basis}\begin{aligned}
\Phi_{\pi,\alpha,\alpha'}(f)&=\sum_{v\in\ONB(\pi)}(\pi(f)\alpha)(v)\alpha'(v)\\
&=\sum_{v\in\ONB(\pi)}\left(\int_Xf(g)\overline{\alpha(\pi(g)v)}\ \d g \right)\alpha'(v).
\end{aligned}
\end{equation}

In the literature, relative characters are initially defined for test functions in $C_c^\infty(X)$ in a slightly different way. Under suitable hypotheses on $\pi$, however, the domain can be extended to larger function spaces. More precisely,  the integral \eqref{equ dual-vector} remains absolutely convergent, and hence $\pi(f)\alpha$ is well defined, in the following situations:
\begin{itemize}
\item $\pi$ is tempered and $f\in\CC(X)$,
\item $\pi$ is square-integrable and $f\in\CC_\temp(X)$,
\item $\pi$ is supercuspidal and $f\in C_\sm(X)$.
\end{itemize}
Moreover, since $f\in C_\sm(X)$ in all three cases, one has $\pi(f)\alpha\in\pi$. Therefore, the relative character can be extended to these larger spaces by the same formula as in (\ref{def rel-char}):
\begin{itemize}
\item $\Phi_{\pi,\alpha,\alpha'}: \CC(X)\ra\BC$ if $\pi$ is tempered;
\item  $\Phi_{\pi,\alpha,\alpha'}: \CC_\temp(X)\ra\BC$ if $\pi$ is square-integrable;
\item  $\Phi_{\pi,\alpha,\alpha'}: C_\sm(X)\ra\BC$ if $\pi$ is supercuspidal.
\end{itemize} 

\begin{thm}\label{thm rel-char}
Suppose that $\pi\in\Irr(\Sp(W))$ is supercuspidal and that $\sigma=\Theta_{V,W,\psi_\tau}(\pi)$ is the first occurrence. Let $f\in \CS(X_{W_F})$ and $f'\in\CS(X_{V_F})$ be test functions that correspond to each other in the sense of Definition \ref{def function}.
 Then, for all $\alpha\in\Hom_{\Sp(W_F)}\left(\pi,\BC\right)$ and  $\beta\in\Hom_{\O(V_F)}\left(\sigma,\BC\right)$, we have
$$\Phi_{\pi,\alpha,T_{V_F}(\beta)}(f)=\Phi_{\sigma,\beta,T_{W_F}(\alpha)}(f').$$
\end{thm}

\begin{proof}  
Pick $\phi\in C_c^\infty(\BY)$ such that $f=p_{V_F}(\phi)$ and $f'=q_{W_F}(\phi)$.
Then $$\begin{aligned}
\Phi_{\pi,\alpha,T_{V_F}(\beta)}(f)&=\sum_{v\in\ONB(\pi)}\left(\int_{X_{W_F}}p_{V_F}(\phi)(g)\overline{\alpha(\pi(g)v)}\ \d g \right) T_{V_F}(\beta)(v)\\
&=\sum_{v\in\ONB(\pi)}T_{W_F}(\alpha)(\theta(\phi,v))\cdot T_{V_F}(\beta)(v)\\
&=\sum_{v\in\ONB(\pi)}\sum_{w\in\ONB(\sigma)}\pair{\theta(\phi,v),w}_\sigma\cdot T_{W_F}(\alpha)(w)\cdot T_{V_F}(\beta)(v).
\end{aligned}$$
On the other hand, 
$$\begin{aligned}
\Phi_{\sigma,\beta,T_{W_F}(\alpha)}(f')&=\sum_{w\in\ONB(\sigma)}\left(\int_{X_{V_F}}q_{W_F}(\phi)(h)\overline{\beta(\sigma(h)w)}\ \d h \right) T_{W_F}(\alpha)(w)\\
&=\sum_{w\in\ONB(\sigma)}T_{V_F}(\beta)(\theta(\phi,w))\cdot T_{W_F}(\alpha)(w)\\
&=\sum_{w\in\ONB(\sigma)}\sum_{v\in\ONB(\pi)}\pair{\theta(\phi,w),v}_\pi\cdot T_{V_F}(\beta)(v)\cdot T_{W_F}(\alpha)(w).
\end{aligned}$$
The desired identity now follows from \eqref{equ inner-equal}.
\end{proof}

\begin{cor}
We retain the assumptions of Theorem \ref{thm rel-char}. Let $\alpha$ be an eigenvector of the normal operator $L_{W_F}$ with eigenvalue $\lambda$, and  set $\beta=T_{W_F}(\alpha)$. Then
$$\Phi_{\sigma,\beta,\beta}(f')=\lambda\Phi_{\pi,\alpha,\alpha}(f).$$
\end{cor}

\begin{proof}
This follows directly from Theorem \ref{thm rel-char}.
\end{proof}

\s{\small Chong Zhang\\
School of Mathematics, Nanjing University,\\
Nanjing 210093, Jiangsu, P. R. China.\\
E-mail address: \texttt{zhangchong@nju.edu.cn}}

\end{document}